\documentclass[11pt]{article}  

\usepackage{amssymb}
\usepackage{bm}
\usepackage{amsfonts}
\usepackage{latexsym}
\usepackage{amsthm}
\usepackage{hyperref}
\usepackage{xcolor} 
\hypersetup{
	colorlinks,
    linkcolor=black,
    citecolor={blue!50!black},
    urlcolor={black},
    runcolor = black
}

\usepackage{enumerate}
\usepackage{epsfig}
\usepackage{graphicx}
\usepackage{color}
\usepackage{float}
\usepackage{subfigure}
\usepackage{amsmath,bm}
\usepackage{amssymb}
\usepackage{stmaryrd}
\usepackage{mathrsfs}
\usepackage{makeidx}
\usepackage{fancyhdr}
\usepackage{lastpage}
\usepackage{url}
\usepackage{geometry}
\usepackage{bbm}
\usepackage{fullpage}
\usepackage{amsmath}
\usepackage{hyperref}
\usepackage{parskip}
\usepackage{booktabs}
\usepackage{multirow}
\usepackage{comment}
\usepackage{mathtools}
\usepackage{slashbox}

\usepackage[titletoc,toc]{appendix}

\restylefloat{table}

\def\qed{\hfill $\Box$}
\renewenvironment{proof}{\vspace{.01cm}   \noindent{\bf Proof }}{\qed \vspace{.1cm}}

\newtheoremstyle{mythmstyle}
{1em} 
{0.1em} 
{\itshape} 
{} 
{\bfseries} 
{} 
{1em} 
{} 

\theoremstyle{mythmstyle}
\newtheorem{theorem}{Theorem}[section]

\newtheorem{lemma}{Lemma}[section]
\newtheorem{proposition}{Proposition}[section]

\newtheorem{remark}{Remark}[section]

\makeatletter \@addtoreset{equation}{section} 

\begin{document}

\title{Analysis of Multi-Index Monte Carlo Estimators\\ for a Zakai SPDE}

\author{Christoph Reisinger\thanks{Mathematical Institute, University of Oxford, Andrew Wiles Building, Woodstock Rd., Oxford OX2, UK,
    {\tt \{reisinge, wangz3\}@maths.ox.ac.uk}.} \and
    Zhenru Wang\footnotemark[1] }
\maketitle

\begin{abstract}
In this article, we propose a space-time Multi-Index Monte Carlo (MIMC) estimator for a one-dimensional parabolic stochastic partial differential equation (SPDE) of Zakai type. We compare the complexity with the Multilevel Monte Carlo (MLMC) method of Giles and Reisinger (2012), and find, by means of Fourier analysis, that the MIMC method: (i) has suboptimal complexity of $O(\varepsilon^{-2}|\log\varepsilon|^3)$ for a root mean square error (RMSE) $\varepsilon$ if the same spatial discretisation as in the MLMC method is used; (ii) has a better complexity of $O(\varepsilon^{-2}|\log\varepsilon|)$ if a carefully adapted discretisation is used; (iii) has to be adapted for non-smooth functionals. Numerical tests confirm these findings empirically.
\end{abstract}

\section{Introduction} 

Stochastic partial differential equations (SPDEs) have become an area of active research over the last few decades. Several classes of methods have been developed to solve SPDEs numerically, including finite difference schemes \cite{gyongy1997implicit, gyongy1998lattice, gyongy1999lattice, davie2001convergence}, finite element schemes \cite{walsh2005finite, kruse2014optimal}, and stochastic Taylor schemes \cite{jentzen2009numerical, jentzen2010taylor}. 

This article is motivated by Zakai SPDEs of the form (see \cite{gobet2006discretization}),
\begin{equation}\label{eq_zakai}
\mathrm{d}v(t,x) = \bigg(\frac{1}{2}\frac{\partial^2}{\partial x^2}\big[a(x)v(t,x)\big] - \frac{\partial}{\partial x}\big[b(x)v(t,x)\big]\bigg)\,\mathrm{d}t - \frac{\partial}{\partial x}\big[\gamma(x)v(t,x)\big]\,\mathrm{d}M_t,
\end{equation}
where $M$ is a standard Brownian motion and $a$, $b$ and $\gamma$ suitably chosen coefficient functions.
This Zakai equation arises from a nonlinear filtering problem: given an observation process $M$ and a signal process $X$, we want to estimate the conditional distribution of $X$ given $M$. If $X$ satisfies
$$X_t = X_0 + \int_0^t\beta(X_s)\,\mathrm{d}s + \int_0^t\sigma(X_s)\,\mathrm{d}B_s + \int_0^t\gamma(X_s)\,\mathrm{d}M_s, $$
where $B$ and $M$ are independent standard Brownian motions, and the distribution function has a density $v$, it is proved in \cite{kurtz1999particle}
that $v$ satisfies \eqref{eq_zakai} with
$$a = \sigma^2 + \gamma^2,\qquad b = \beta.$$
The conditional (on $M$) distribution function is then
\[
L_t^x = \int_{-\infty}^x v(t,\xi)\,\mathrm{d}\xi = 1-\int^{\infty}_x v(t,\xi) \,\mathrm{d}\xi,
\]
and it is the goal of this article to estimate the expectation of functionals of this form.

For simplicity, we restrict ourselves to the special case
\begin{equation}\label{eq_equation1}
\mathrm{d}v = -\mu\frac{\partial v}{\partial x}\,\mathrm{d}t + \frac{1}{2}\frac{\partial^2 v}{\partial x^2}\,\mathrm{d}t - \sqrt{\rho}\frac{\partial v}{\partial x}\,\mathrm{d}M_t,\qquad (t,x)\in(0,T)\times\mathbb{R},
\end{equation}
where $T>0$, $M$ is a standard Brownian motion, and $\mu$ and $0\le \rho<1$ are real-valued parameters. 
This is a special case of (\ref{eq_zakai}) where $\sigma = \sqrt{1-\rho},\ \gamma = \sqrt{\rho},\ \beta=\mu$.

Moreover, $v$ in (\ref{eq_equation1}) describes the limit empirical measure, as $N\rightarrow\infty$, of a large exchangeable particle system \cite{kurtz1999particle},
\begin{equation}
\mathrm{d}X_t^i = \mu\,\mathrm{d}t + \sqrt{1-\rho}\,\mathrm{d}W_t^i + \sqrt{\rho}\,\mathrm{d}M_t,\qquad\text{for }1\leq i\leq N,
\end{equation}
where $W_t^i$ and $M_t$ are independent standard Brownian motions.

A direct application of this model is the large portfolio credit model \cite{bush2011stochastic}. Assume the market consists of $N$ different firms where $X_t^{i}$ are ``distance-to-default'' processes.
Then the functional of interest is the loss
\begin{equation}
\label{eq_1dlossfunctional}
L_t =  1-\int_0^\infty v(t,x)\,\mathrm{d}x,
\end{equation}
i.e., the mass lost at the absorbing boundary. 
In the credit risk application of \cite{bujok2012numerical}, $L$ describes the loss in a structural credit model, i.e., the fraction of firms whose values have crossed zero and which are considered defaulted. The values of credit products are often functions of the loss $L_t$.

Generally, the solution to (\ref{eq_zakai}) is not known analytically and has to be approximated numerically. A survey of methods is given in \cite{gobet2006discretization}, and we focus here on recent applications of multilevel methods as they pertain to this article. 
Giles and Reisinger \cite{ref2} used an explicit Milstein finite difference approximation to the solution of (\ref{eq_equation1}).
By using Fourier analysis, this scheme can be shown to give first order of convergence in the timestep and second order in the spatial mesh size. One constraint in this paper is that the timestep needs to be small enough to ensure stability. Inspired by the numerical analysis of SDEs in \cite{szpruch2010numerical,buckwar2011comparative}, \cite{ref1} extended the discretisation to an implicit method on the basis of the $\sigma$-$\theta$ time-stepping scheme, where the drift and the deterministic part of the double stochastic integral are taken implicit. Fourier analysis shows that the convergence order is the same as in the explicit Milstein scheme, however, this scheme is unconditionally mean-square stable under a constraint on the correlation $\rho$ in~\eqref{eq_equation1}. 
This unconditional stability is essential for our application as detailed below.

In this paper, we compare a new Multi-index Monte Carlo (MIMC) scheme in the spirit of \cite{ref5}, with the Multilevel Monte Carlo (MLMC) method of \cite{giles2008multilevel}.
The MLMC method utilises a sequence of approximations $P_0,P_1,\ldots,P_{l^\ast}$ to a random variable $P$ with increasing accuracy but also higher cost for increasing $l$. In the simulation of an SDE, $l$ would typically be the refinement level of the time mesh, with $2^l$ time steps.
The MLMC estimator is based on recursive control variates
embedded in the identity
$$\mathbb{E}\left[P_{l^\ast}\right] = \mathbb{E}\big[P_0\big] + \sum_{l=1}^{l^\ast}\mathbb{E}\big[P_l-P_{l-1}\big],$$
where $l^\ast$ is a maximum refinement level. The goal is to estimate $\mathbb{E}[P_{l^\ast}]$ by independent estimation of the summands, in a way that the root mean square error (RMSE) is comparable to the bias, but with a much reduced computational complexity. If fewer samples are needed for higher levels, the total computational cost is much lower than when using the standard Monte Carlo method (see Theorem 1 in \cite{giles2015multilevel}). In a second step, $l^\ast$ is adapted for a given RMSE target.

The MLMC method has been extended to SPDEs, in which case it gives even better savings due to the additional spatial dimensions. Giles and Reisinger \cite{ref2} applied MLMC to simulate the SPDE~\eqref{eq_equation1} and the functional of its solution given in~\eqref{eq_1dlossfunctional}. Both timestep $k$ and mesh size $h$ decrease geometrically on different levels of refinement $l=0,1,\ldots,l^\ast$, with fixed $k/h^2$. Thus, the variance of the estimators for $\mathbb{E}\big[P_l-P_{l-1}\big]$ is small for large $l$, compared to the estimator for $\mathbb{E}\big[P_{l^\ast}\big]$. As a result, for a fixed accuracy $\varepsilon$, the cost can be reduced significantly to $O(\varepsilon^{-2})$ instead of $O(\varepsilon^{-7/2})$ by the standard Monte Carlo estimator.

The approach taken here is to approximate this SPDE \eqref{eq_equation1} by multi-index Monte Carlo simulation as introduced in \cite{ref5}. 
The goal of MIMC is to improve the total complexity for higher-dimensional problems, where the cost of each sample on higher levels increases faster than the variance of MLMC estimators decays, such that the optimal complexity of $O(\varepsilon^{-2})$ is lost in MLMC. Multi-index Monte Carlo can be viewed as a combination of sparse grids \cite{zenger1997sparse} and MLMC methods. The level $\bm{l}$ is now a vector of integer indices, corresponding to the refinement level in the different dimensions.
Sparse grids, and hence MIMC, use high order mixed differences instead of only taking first order differences to construct a hierarchical approximation. More specifically, for a $d$-dimensional problem, \cite{ref5} first defines first order difference operators $\Delta_i,\ 1\leq i\leq d$, such that
$$\Delta_i P_{\bm{l}} = P_{\bm{l}} - P_{\bm{l}-\bm{e}_i}\mathbbm{1}_{l_i>0},$$
where $\bm{e}_i$ is the unit vector in direction $i$ and $P_{\bm{l}}$ is the  estimator of level ${\bm{l}}$. Then the telescoping sum becomes
$$\mathbb{E}[P] = \sum_{\bm{l} \in \mathbb{N}^d}\mathbb{E}\bigg[\bigg(\prod_{i=1}^d\Delta_i\bigg)P_{\bm{l}}\bigg].$$
The strategy is to neglect terms with large ${\bm{l}}$ where the cost is large and the contribution to the estimator is small. Instead of choosing an index $l^\ast$ as in MLMC, one now has to choose an  index set $\mathcal{I} \subset \mathbb{N}^d$ of terms to include.

The main selling point of MIMC methods applied to higher-dimensional SPDEs, similar to sparse grid methods for high-dimensional PDEs, is that they can give a computational complexity with a fixed polynomial order multiplied by a log term, with an exponent which grows with the dimension $d$. 
Both the sparse grid combination technique and the MIMC framework are based on certain specific expansions of the discretisation error in the mesh parameters. These have been proven for a growing class of PDEs \cite{pflaum1999error, reisinger2012analysis, griebel2014convergence}; \cite{ref5} studies PDEs with random coefficients and refers to these PDE results in a path-wise sense, \cite{giles2015multilevel} analyses MIMC for nested simulation.

In this paper, we
\begin{itemize}
\item
apply MIMC to the SPDE \eqref{eq_equation1} on a space-time mesh with an implicit Milstein time-stepping scheme and central spatial differences;
\item
demonstrate that the expectation and variance of the multi-index estimators can be analysed by Fourier analysis, even if the required error expansions do not hold pathwise;
\item
show that an extra term appears in the error expansion which explodes for small spatial mesh sizes, but by a suitable adaptation of the MIMC estimator a theoretical complexity of $O(\varepsilon^{−2}|\log \varepsilon|)$ can be obtained, while the practically observed complexity is still $O(\varepsilon^{−2})$ as in the MLMC method \cite{ref2};
\item
give a seemingly innocuous variant of this approximation scheme, i.e., with the spatial approximation studied in \cite{ref2}, and show that the ``standard'' assumptions on the error expansion are only satisfied with lower order, such that the MIMC estimator is less efficient with complexity $O(\varepsilon^{-2}|\log\varepsilon|^3)$;
\item
give an example of a discontinuous functional and a simple approximation where we show that multiple leading order error terms appear, such that the optimal index set is not triangular, and demonstrate how MIMC can be adapted to this case.
\end{itemize}

Hence we show that although
MIMC can be expected to work well in multi-dimensional situations compared to MLMC, this is not always true. In this paper, we only consider cases where MLMC gives optimal complexity order. Therefore, the results herein can be considered negative in the practical sense that worse performance, is obtained for MIMC than MLMC. In addition to these ``negative'' results being interesting in their own right, we provide a methodology to analyse MIMC estimators for a class of SPDEs. This method carries over to the higher-dimensional setting, where MIMC has complexity advantages over MLMC (as explained further in Section \ref{sec_Conclusion}).

The rest of this article is structured as follows. We define two implicit Milstein finite difference schemes in Section~\ref{sec_1dImplicitMilsteinScheme}, and the MIMC estimators in Section~\ref{sec_1dMIMCEstimator}. The main theoretical result is provided in Section \ref{sec_1dMIMCFourier}, where we analyse the convergence order of a MIMC estimator using Fourier analysis. Section \ref{sec_1dOptimalIndex} gives an ``optimal index set'' for this MIMC estimator and derives the total complexity for fixed accuracy. Section \ref{sec_1dOtherResults} derives further results for alternative approximations, while Section \ref{sec_1dNumerical} shows numerical experiments confirming the above findings. Section \ref{sec_Conclusion} offers conclusions and directions for further research.

\section{Two implicit Milstein finite difference schemes}\label{sec_1dImplicitMilsteinScheme}
Let $(\Omega,\mathcal{F},\mathbb{P})$ be a probability space, on which there is given a one-dimensional standard Brownian motion $M$. We study the parabolic stochastic partial differential equation
\begin{equation}\label{eq_1dSPDE}
\mathrm{d}v = -\mu\frac{\partial v}{\partial x}\,\mathrm{d}t + \frac{1}{2}\frac{\partial^2 v}{\partial x^2}\,\mathrm{d}t - \sqrt{\rho}\frac{\partial v}{\partial x}\,\mathrm{d}M_t,\qquad (t,x)\in(0,T)\times\mathbb{R},
\end{equation}
where $T>0$, $M$ is a standard Brownian motion, and $\mu$ and $0\le \rho<1$ are real-valued parameters,
subject to the Dirac initial data
\begin{equation}\label{eq_1dDiracInitial}
v(0,x) = \delta(x-x_0),
\end{equation}
where $x_0\in\mathbb{R}$ is given. A classical result states that, for a class of SPDEs including \eqref{eq_1dSPDE}, with initial condition in $L_2(\mathbb{R}) = \{f:\int_{\mathbb{R}}f^2\,\mathrm{d}x<\infty\}$, there exists a unique solution $v\in L_2(\Omega\times(0,T), \mathcal{F}, L_2(\mathbb{R}))= \{f(\omega,t,\cdot):\, f(\omega,t,\cdot)\in L_2(\mathbb{R}),\ f(\omega,t,\cdot) \text{ is } \mathcal{F}_t\text{-measurable, }\mathbb{E}\int_0^T \lVert f(\omega,t)\rVert^2_{L_2}\,\mathrm{d}t<\infty\}$ \cite{krylov1981stochastic}. This does not include Dirac initial data \eqref{eq_1dDiracInitial}, but in fact, the solution to \eqref{eq_1dSPDE} and \eqref{eq_1dDiracInitial} at time $T$ is analytically known (see \cite{ref2}) to be the smooth (in $x$) function
\begin{equation}\label{eq_1dTheoreticalResult}
v(T,x) = \frac{1}{\sqrt{2\pi(1-\rho)T}}\exp\left(-\frac{\big(x-x_0-\mu T-\sqrt{\rho}M_T\big)^2}{2(1-\rho)T}\right).
\end{equation}

Integrating \eqref{eq_1dSPDE} over a time interval $[t,t+k]$, we have
$$v(t+k,x) = v(t,x)+\int_t^{t+k}-\mu\frac{\partial v}{\partial x} + \frac{1}{2}\frac{\partial^2 v}{\partial x^2}\;\mathrm{d}s-\int_t^{t+k}\sqrt{\rho}\frac{\partial v}{\partial x}\;\mathrm{d}M_s.$$

We then use a spatial grid with uniform spacing $h$ and a timestep $k$, and let $V_j^{n}$ be the approximation to $v\big(nk,jh\big)$, $n=1,\ldots,N$, $j\in\mathbb{Z}$, where $N=T/k$. We approximate $v(0,x)$ by
\begin{equation}\label{eq_1dDiracInitialApprox}
V_j^0 = 
\begin{cases}
h^{-1},\quad &j=[x_0/h],\\
0,&\text{otherwise}.
\end{cases}
\end{equation}
Here $[x]$ denotes the closest integer to $x$. 
To improve the accuracy of the approximation of $v$ in the present case of Dirac initial data, we subsequently choose $h$ such that $x_0$ is on the grid, then $x_0/h$ is an integer.

By an explicit Milstein scheme \cite{ref2} with standard central difference approximation to the spatial derivatives, we have
$$V^{n+1} = V^n - \frac{\mu k+\sqrt{\rho k}Z_n}{2h}D_1V^n + \frac{(1-\rho)k + \rho kZ_n^2}{2h^2}D_2V^n,$$
where 
$Z_n$ are independent standard normal random variables, $D_1$ and $D_2$ are first and second central difference operators,
$$(D_1V)_j = V_{j+1}-V_{j-1},\quad (D_2V)_j = V_{j+1}-2V_j+V_{j-1}.$$
However, there is a condition for the mean-square stability of this explicit scheme (see \cite{ref2}),
$\left(1+2\rho^2\right)\frac{k}{h^2}\leq~1.$
This is an obstacle for the use of the MIMC scheme where timestep and mesh size are varied independently. We can avoid the constraint on the timestep by using an implicit scheme instead \cite{ref1}.
Two implicit Milstein finite difference discretisations are conceivable for the SPDE \eqref{eq_1dSPDE}:
\begin{enumerate}
\item Discretise the spatial derivatives first, then apply the Milstein scheme to the resulting system of SDEs, i.e.,
\begin{equation}\label{eq_1dimplicitdiscrete}
V^{n+1} = V^n - \frac{\mu k}{2h}D_1V^{n+1} + \frac{k}{2h^2}D_2V^{n+1} - \frac{\sqrt{\rho k}Z_n}{2h}D_1V^n + \frac{\rho k(Z_n^2 - 1)}{8h^2}D_1^2V^n,
\end{equation}
where $(D_1^2V)_j = V_{j+2}-2V_j+V_{j-2}$. 
\item Apply the Milstein scheme first, then discretise the spatial derivatives. This gives
\begin{equation}\label{eq_1dimplicitdiscrete_2}
V^{n+1} = V^n - \frac{\mu k}{2h}D_1V^{n+1} + \frac{k}{2h^2}D_2V^{n+1} - \frac{\sqrt{\rho k}Z_n}{2h}D_1V^n + \frac{\rho k(Z_n^2 - 1)}{2h^2}D_2V^n.
\end{equation}
\end{enumerate}
These two schemes have the same effect on the multilevel algorithm, and \eqref{eq_1dimplicitdiscrete_2} is used in \cite{ref2}. However, we will show that \eqref{eq_1dimplicitdiscrete_2} is less efficient using the multi-index algorithm, as these two schemes result in different orders of variance for the MIMC estimators. Since the scheme \eqref{eq_1dimplicitdiscrete} is more efficient, we will explore it in detail in Section \ref{sec_1dMIMCFourier}, and the analysis of \eqref{eq_1dimplicitdiscrete_2} in Section~\ref{sec_1dalternative} is then similar. 

Let $\ell_2(\mathbb{Z}) = \{(\ldots,V_{-1},V_0,V_1,\ldots): (\sum_{n=-\infty}^\infty |V_n|^2)^{\frac{1}{2}}<\infty\}$ and $L_2(\Omega\times(0,T),\mathcal{G},\ell_2(\mathbb{Z})) = \{V(\omega,t,\cdot): V(\omega,t,\cdot)\in \ell_2(\mathbb{Z}),V(\omega,t,\cdot) \text{ is }\mathcal{F}_t\text{-measurable, }\mathbb{E}\int_0^T \lVert V(\omega,t)\rVert^2_{\ell_2}\,\mathrm{d}t<\infty\}$.

The implicit schemes are unconditionally stable and converge in mean-square sense ($\ell_2$ sense) under a constraint on the correlation \cite{ref1}, as the following theorems describe for \eqref{eq_1dimplicitdiscrete_2}, and an analogous result holds for  \eqref{eq_1dimplicitdiscrete}.
\begin{theorem}[Theorem 2.1 in \cite{ref1}]\label{thm_1dImplicitStability}
The implicit Milstein scheme \eqref{eq_1dimplicitdiscrete_2} is unconditionally stable in the mean-square sense, provided $0<\rho\leq 1/\sqrt{2}$.
\end{theorem}

\begin{theorem}[Theorem 2.2 in \cite{ref1}]\label{thm_1dImplicitConvergence}
Assume $0<\rho\leq 1/\sqrt{2},\ T>0,\ k=T/N$ and $\frac{k}{h^2}=\lambda>0$ is kept fixed. The implicit Milstein scheme \eqref{eq_1dimplicitdiscrete_2} has the error expansion, for Dirac initial data \eqref{eq_1dDiracInitial},
\[
V_j^N - v(T,x_j) = k\,E_1(T,x_j) + h^2\,E_2(T,x_j) + o(k,h^2)\,R(T,x_j),
\]
where $x_j = jh$, $E_1,\,E_2$ and $R$ are random variables with bounded moments.
\end{theorem}

\begin{proposition}\label{prop_1dImplicitConvergence}
Under the conditions of Theorem~\ref{thm_1dImplicitConvergence} but without fixed $k/h^2$, then for $h^2\geq k$, Theorem~\ref{thm_1dImplicitConvergence} still holds; for $h^2<k$, the implicit Milstein scheme \eqref{eq_1dimplicitdiscrete} has the error expansion
\begin{equation}\label{eq_beginningoftheta}
V_j^N - v(T,x_j) = k\,E_1(T,x_j) + h^2\,E_2(T,x_j) + \theta^{N}h^{-1} \,E_3(T,x_j)+ o(k,h^2)\,R(T,x_j),
\end{equation}
where $N = T/k,\, 0<\theta<1$, and $\theta$ is independent of $h$ and $k$.
\end{proposition}
\begin{proof}
See Appendix \ref{appendix_proof_prop_meansquareconvergence}.
\end{proof}

\begin{remark}\label{rmk_1dImplicitConvergence}
If for some $\beta>0$, and a constant $C_0>0$ independent of $h$ and $k$,
\begin{equation}\label{eq_1d_k/hcondition}
\theta^{\frac{T}{k}} \leq 2^{-C_0}\cdot h^{3+\beta},
\end{equation}
or equivalently,
\begin{equation}\label{eq_1d_k/hcondition_2}
k\leq \frac{T\log_2(\theta^{-1})}{C_0+(3+\beta)\log_2(h^{-1})},
\end{equation}
then the implicit Milstein scheme \eqref{eq_1dimplicitdiscrete} has the error expansion
\[
V_j^N - v(T,x_j) = k\,E_1(T,x_j) + h^2\,E_2(T,x_j)+ o(k,h^2)\,R(T,x_j).
\]
\end{remark}

\begin{remark}
Theorem~\ref{thm_1dImplicitStability}, Theorem~\ref{thm_1dImplicitConvergence}, Proposition~\ref{prop_1dImplicitConvergence} and Remark~\ref{rmk_1dImplicitConvergence} are valid for both implicit Milstein schemes \eqref{eq_1dimplicitdiscrete} and \eqref{eq_1dimplicitdiscrete_2}.
\end{remark}

We consider a specific linear functional of $v$ for fixed $T$, the random variable
\begin{equation}\label{eq_1dLossTheoretical}
L = \int_{-\infty}^0 v(T,x)\,\mathrm{d}x = 1-\int_0^\infty v(T,x)\,\mathrm{d}x,
\end{equation}
as discussed in the introduction, where $v$ is the solution to (\ref{eq_1dSPDE}) and (\ref{eq_1dDiracInitial}).

By introducing integer multi-indices $\bm{l} = (l_1,l_2)\in\mathbb{N}^2$ as the index of space and time separately, we denote $L_{(l_1,\,l_2)}$ as the discrete approximations to $L$, with mesh size $h = h_0\cdot2^{-l_1}$, and timestep $k = k_0\cdot2^{-2l_2}$,
\begin{equation}\label{eq_1dLossApprox}
L_{(l_1,\,l_2)} = 1-h\sum_{j=1}^{\infty} V_{j}^N - \frac{h}{2}V_0^N,
\end{equation}
where the integral is approximated by the trapezoidal rule and $V_j^N$ is determined by (\ref{eq_1dimplicitdiscrete}). 

\begin{proposition}\label{prop_1dErrorofLoss}
Let $L: \Omega\rightarrow \mathbb{R}^+$ be the random variable given by \eqref{eq_1dLossTheoretical}. Let $L_{(l_1,\,l_2)}:\Omega\rightarrow \mathbb{R}^+$ be the approximation to $L$ given by \eqref{eq_1dLossApprox}. 
Assume $0<\rho\leq 1/\sqrt{2}$. 
Then there exists a real number $C>0$, such that for any $l_1^\ast\in\mathbb{N}$, any $h_0>0$, and any $k_0>0$ such that
\[
k_0\leq \frac{T\log_2(\theta^{-1})}{C_0+(3+\beta)(l_1^\ast + \log_2h_0^{-1})}\,,
\]
where $0<\theta<1$, $\beta>0$, $C_0>0$ are the same constants as in Remark \ref{rmk_1dImplicitConvergence}, the following holds: 

For any $0\leq l_1 \leq l_1^\ast$, $l_2\geq0$,
\[
\sqrt{\mathbb{E}\left[\lvert L_{(l_1,\,l_2)} - L\rvert^2\right]} \leq C\left(h^2+k\right),
\]
where $h = h_0\cdot 2^{-l_1},\,k = k_0\cdot 2^{-2l_2}$.
\end{proposition}
\begin{proof}
Denote the trapezoidal approximation of $1-\int_0^\infty v(T,x)\,\mathrm{d}x$ with mesh size $h$ by
$$\overline{L} = 1-h\sum_{j=1}^{\infty}v(T,x_j) - \frac{h}{2}v(T,0),$$
where $x_j = jh$.
Since both $L$ and $\overline{L}$ use the theoretical, smooth $v(T,x)$, we have
\begin{align*}
\left|\overline{L}-L\right|&= \left| \sum_{j=1}^\infty\left[ \int_{x_{j-1}}^{x_j} v(T,x)\,\mathrm{d}x - \frac{1}{2}h\left( v(T,x_{j-1}) + v(T,x_j) \right) \right]\right|\\
&\leq \frac{1}{12}h^3\sum_{j=1}^\infty \max_{\zeta\in[x_{j-1},x_j]}\left|\frac{\partial^2 v}{\partial x^2}(T,\zeta)\right|.
\end{align*}

Given $v(T,x)$ from \eqref{eq_1dTheoreticalResult},
\[
v(T,x) = \frac{1}{\sqrt{2\pi(1-\rho)T}}\exp\left(-\frac{\big(x-x_0-\mu T-\sqrt{\rho}M_T\big)^2}{2(1-\rho)T}\right),
\]
for all $h_0$, there exists a constant $C_0$, such that for $h\leq h_0$,
\[
\mathbb{E}\left[\left|h\sum_{i=1}^\infty \max_{\zeta\in[x_{j-1},x_j]}\left|\frac{\partial^2 v}{\partial x^2}(T,\zeta)\right|\right|^2\right] \leq C_0.
\]

Therefore
\[
\mathbb{E}\left[\left\lvert \overline{L} - L\right\rvert^2\right]\leq \frac{1}{144}C_0\,h^4.
\]

From the proof of Theorem 2.2 in \cite{ref1}, there exist random variables $\widetilde{C}_j$, and a constant $C_1>0$, all of which are independent of $(l_1,l_2)$, such that 
\[
\Big| V_j^N - v(T,x_j) \Big| \leq \widetilde{C}_j (h^2+k),
\]
where
\[
\mathbb{E}\bigg[ \bigg|h\sum_{j=0}^{\infty} \widetilde{C}_j\bigg|^2 \bigg]< C_1 < \infty.
\]

Combining above equations, we have
\begin{align*}
&\quad\mathbb{E}\left[\left\lvert L_{(l_1,\,l_2)} - L\right\rvert^2\right] = \mathbb{E}\left[\left\lvert L_{(l_1,\,l_2)}-\overline{L} + \overline{L} - L\right\rvert^2\right]\leq 2\, \mathbb{E}\left[\left\lvert L_{(l_1,\,l_2)} - \overline{L}\right\rvert^2\right] + 2\,\mathbb{E}\left[\left\lvert \overline{L} - L\right\rvert^2\right]\\
&= 2\mathbb{E}\left[\bigg| h\sum_{j=1}^{\infty}\big(v(T,x_j)-V_j^N\big) + \frac{h}{2}\big(v(T,0) - V_0^N\big)\bigg|^2\right] + 2\,\mathbb{E}\left[\left\lvert \overline{L} - L\right\rvert^2\right]\\
&\leq 4\big(h^4+ k^2\big) \mathbb{E}\left[ \bigg|h\sum_{j=0}^\infty \widetilde{C}_j\bigg|^2 \right] + \frac{1}{72}C_0\,h^4.
\end{align*}
Therefore, by choosing $C = \sqrt{\frac{1}{72}C_0+4C_1}$, we have for all $0\leq l_1 \leq l_1^\ast$, $l_2\geq0$,
$$\sqrt{\mathbb{E}\left[\lvert L_{(l_1,\,l_2)} - L\rvert^2\right]} \leq C\left(h^2+k\right).$$
\end{proof}

\section{Multi-index Monte Carlo discretisation on a space-time mesh}\label{sec_1dMIMCEstimator}

Following the idea in \cite{ref5}, we introduce $\Delta_i$, the first order difference operator along directions $i=1,2$, defined as
\begin{equation}\label{eq_1dfirstorderdifference}
\Delta_iL_{\bm{l}} = 
\begin{cases}
L_{\bm{l}}-L_{\bm{l}-\bm{e}_i},\quad &\text{if}\ l_i>0,\\
L_{\bm{l}}&\text{if}\ l_i=0,
\end{cases}
\end{equation}
with $\bm{e}_i$ being the canonical vectors in $\mathbb{R}^2$, i.e., $(\bm{e}_i)_j = \delta_{ij}$. Then define the first order mixed difference operator $\Delta = \Delta_1\otimes\Delta_2$. Hence, for $l_1>0,\ l_2>0$, we have
\begin{equation}\label{eq_1dDeltaLinDetails}
\begin{aligned}
\Delta L_{(l_1,\,l_2)} &= \left(L_{(l_1,\,l_2)}-L_{(l_1,\,l_2-1)}\right)-\left(L_{(l_1-1,\,l_2)}-L_{(l_1-1,\,l_2-1)}\right)\\
&= L_{(l_1,\,l_2)}-L_{(l_1,\,l_2-1)}-L_{(l_1-1,\,l_2)}+L_{(l_1-1,\,l_2-1)}.
\end{aligned}
\end{equation}

A telescoping sum then gives, for any $l_1^\ast,\,l_2^\ast \ge 0$,
\[
L_{(l_1^\ast,\,l_2^\ast)} = \sum_{l_1=0}^{l_1^\ast}\sum_{l_2=0}^{l_2^\ast} \Delta L_{(l_1,\,l_2)},\quad\quad \mathbb{E}\big[L_{(l_1^\ast,\,l_2^\ast)}\big] = \sum_{l_1=0}^{l_1^\ast}\sum_{l_2=0}^{l_2^\ast} \mathbb{E}\big[\Delta L_{(l_1,\,l_2)}\big].
\]
By setting $\mathcal{I}_0=\{(l_1,l_2)\in\mathbb{N}^2: l_1\leq l_1^\ast,\,l_2\leq l_2^\ast \}$ and $L_{\mathcal{I}_0} = \sum_{(l_1,\,l_2)\in\mathcal{I}_0}\Delta L_{(l_1,\,l_2)}$,
with $h_{l^\ast},\,k_0$ satisfy the condition \eqref{eq_1d_k/hcondition}. Then the weak error between $L_{(l_1^\ast,\,l_2^\ast)}$ can be formed as 
\begin{equation}\label{eq_1d_weakerror0}
\mathbb{E}\big[L_{(l_1^\ast,\,l_2^\ast)}-L\big] = \sum_{l_1=0}^{l_1^\ast}\sum_{l_2=0}^{l_2^\ast} \mathbb{E}\big[\Delta L_{(l_1,\,l_2)}\big]<\varepsilon_0.
\end{equation}

However, this is not necessarily the most efficient way of approximating $L$, since the approximations $L_{(l_1,\,l_2)}$ with $(l_1,\,l_2)$ close to $(l_1^\ast,\,l_2^\ast)$ are highly accurate but also costly to compute. Instead, we only compute the levels in a minimal index set $\mathcal{I}\subset\mathcal{I}_0$ (adapted to the convergence in $h$ and $k$) which satisfies
\begin{equation}\label{eq_1d_weakerror1}
\Bigg|\sum_{(l_1,\,l_2)\in\mathcal{I}_0\setminus\mathcal{I}}\mathbb{E}\big[\Delta L_{(l_1,\,l_2)}\big]\Bigg|<\varepsilon_1,
\end{equation}
for a given error $\varepsilon_1$. Then it can be easily proved that
\begin{equation}\label{eq_1d_weakerror}
\Bigg|\sum_{(l_1,\,l_2)\in\mathcal{I}}\mathbb{E}\big[\Delta L_{(l_1,\,l_2)}\big] -\mathbb{E}\big[L\big]\Bigg|\leq\varepsilon_0+\varepsilon_1.
\end{equation}

The MIMC estimator is now defined by (see \cite{ref5})
\begin{equation}\label{eq_1dMIMCestimator}
\widehat{L}_{\mathcal{I}}\coloneqq\sum_{(l_1,\,l_2)\in\mathcal{I}}\frac{1}{M_{(l_1,\,l_2)}}\sum_{m=1}^{M_{(l_1,\,l_2)}}\Delta L_{(l_1,\,l_2)}^{(m)},
\end{equation}
where $\mathcal{I}\subset\mathbb{N}^2$ is an index set, and $M_{(l_1,l_2)}$ is the number of samples for each $(l_1,l_2)\in\mathcal{I}$. The key point in the algorithm is that the quantity $\Delta L_{(l_1,\,l_2)}$ comes from four discrete approximations using the same Brownian path, therefore the variance of $\Delta L_{(l_1,\,l_2)}$ is small, in a sense made precise in Theorem \ref{thm_1dmainresult}. 
Thus the moments of $\Delta L_{(l_1,\,l_2)}$ are small not only if both $k$ and $h$ are small, but it suffices that either $k$ or $h$ is small. This allows the omission of computationally costly indices.
The MIMC algorithm is based on a good choice of $\mathcal{I}$ and, for given $(l_1,\,l_2)\in\mathcal{I}$, $M_{(l_1,\,l_2)}$ to balance bias and variance for a given accuracy target.

Denote $E_{(l_1,\,l_2)}\coloneqq\left|\mathbb{E}\left[\Delta L_{(l_1,\,l_2)}\right]\right|$, $V_{(l_1,\,l_2)}\coloneqq\mathrm{Var}\big[\Delta L_{(l_1,\,l_2)}\big]$, and $W_{(l_1,\,l_2)}$ the average work required for a realization of $\Delta L_{(l_1,\,l_2)}$ for a single path. Then the total work corresponding to the estimator is
\begin{equation}\label{eq_1dTotalwork}
W = \sum_{(l_1,\,l_2)\in\mathcal{I}}W_{(l_1,\,l_2)}M_{(l_1,\,l_2)}.
\end{equation}
As the paths are chosen independently for different $(l_1,\,l_2)$, the variance of the estimator is
\begin{equation}\label{eq_1dvarestimator}
\mathrm{Var}\big[\widehat{L}_{\mathcal{I}}\big] = \sum_{(l_1,\,l_2)\in\mathcal{I}}\frac{V_{(l_1,\,l_2)}}{M_{(l_1,\,l_2)}}.
\end{equation}

The mean square error can be expressed as the sum of two contributions, bias and variance,
\begin{equation}\label{eq_1drms}
\mathbb{E}\Big[\big(\widehat{L}_{\mathcal{I}}-\mathbb{E}[L]\big)^2\Big] = \mathrm{Var}\big[\widehat{L}_{\mathcal{I}}\big] + \Big(\mathbb{E}\big[\widehat{L}_{\mathcal{I}}\big]-\mathbb{E}[L]\Big)^2.
\end{equation}
To achieve a root mean square error of $\varepsilon$, we split the accuracy as follows,
\begin{align*}
\big|\mathbb{E}\big[\widehat{L}_{\mathcal{I}}-L\big]\big| &\leq \alpha\varepsilon, \qquad \mathrm{Var}\big[\widehat{L}_{\mathcal{I}}\big] \leq \big(1-\alpha^2\big)\varepsilon^2,
\end{align*}
where $0<\alpha< 1$. By optimising the total work with respect to $M_{\bm{l}}$ given the variance constraint~\eqref{eq_1dvarestimator} and a fixed $\mathcal{I}$, we can derive (see also \cite{ref5}) the optimal number of samples for each level $\bm{l} = (l_1,\,l_2)$ to be
\begin{equation}\label{eq_M}
M_{\bm{l}} = (1-\alpha^2)^{-2}\varepsilon^{-2}\bigg(\sum_{\bm{l}\in\mathcal{I}}\sqrt{V_{\bm{l}}W_{\bm{l}}}\bigg)\sqrt{\frac{V_{\bm{l}}}{W_{\bm{l}}}}.
\end{equation}
In the numerical implementation, we take the integer ceiling of $M_{(l_1,\,l_2)}$ in \eqref{eq_M}, as a result we assume the bound
\begin{align*}
M_{\bm{l}} \leq 1 + (1-\alpha^2)^{-2}\varepsilon^{-2}\bigg(\sum_{\bm{l}\in\mathcal{I}}\sqrt{V_{\bm{l}}W_{\bm{l}}}\bigg)\sqrt{\frac{V_{\bm{l}}}{W_{\bm{l}}}},\qquad \forall\, \bm{l}\in\mathcal{I}.
\end{align*}
Therefore the total work is bounded by
\begin{equation}\label{eq_totalwork}
W\leq (1-\alpha^2)^{-2}\varepsilon^{-2}\bigg(\sum_{\bm{l}\in\mathcal{I}}\sqrt{V_{\bm{l}}W_{\bm{l}}}\bigg)^2 + \sum_{\bm{l}\in\mathcal{I}}W_{\bm{l}},
\end{equation}
where the second term is usually negligible.

In \cite{ref5}, Theorem 2.2 shows the total computational cost using an optimal index set with the MIMC method given bounds on $E_{(l_1,\,l_2)}$ and $V_{(l_1,\,l_2)}$. Based on these ideas, we first consider the Fourier analysis of each level $\Delta L_{(l_1,\,l_2)}$, then find an index set $\mathcal{I}$ to achieve the optimal order of the total work. Though the complexity is the same for both schemes \eqref{eq_1dimplicitdiscrete} and \eqref{eq_1dimplicitdiscrete_2} when using the multilevel method, we will see that it is different using the multi-index method.  

\section{Fourier analysis of MIMC estimators}\label{sec_1dMIMCFourier}
In this section, we show that the MIMC increments \eqref{eq_1dDeltaLinDetails} for scheme \eqref{eq_1dimplicitdiscrete} satisfy the moment conditions required for the analysis of \cite{ref5}. We then derive the optimal index set and the resulting complexity of the MIMC estimator in Section \ref{sec_1dOptimalIndex}. In what follows, we will use $C$ to stand for generic positive real constants dependent only on model parameters, and their values may change between occurrences.

\subsection{Main results}
The following theorem is concerned with the first and second moments of $\Delta L_{(l_1,\,l_2)}$ with schemes \eqref{eq_1dimplicitdiscrete} and \eqref{eq_1dLossApprox}. 
\begin{theorem}\label{thm_1dmainresult}
Consider $L_{(l_1,\,l_2)}$ from \eqref{eq_1dLossApprox} with the implicit Milstein scheme~\eqref{eq_1dimplicitdiscrete}. Assume $0<\rho\leq 1/\sqrt{2}$. 
Then there exists a real number $C>0$, such that for any $l_1^\ast\in\mathbb{N}$, any $h_0>0$, and any $k_0>0$ such that
\[
k_0\leq \frac{T\log_2(\theta^{-1})}{C_0+(3+\beta)(l_1^\ast + \log_2h_0^{-1})}\,,
\]
where $0<\theta<1$, $\beta>0$, $C_0>0$ are the same constants as in Remark \ref{rmk_1dImplicitConvergence}, the following holds:

For any $0\leq l_1 \leq l_1^\ast$, $l_2\geq0$, the first and second moments of $\Delta L_{(l_1,\,l_2)}$ satisfy
\begin{equation}
\big|\mathbb{E}\left[\Delta L_{(l_1,\,l_2)}\right]\big| \leq C\,h^2k,\qquad
\mathbb{E}\big[\left|\Delta L_{(l_1,\,l_2)}\right|^2\big] \leq C\,h^4k^2,
\end{equation}
where $h=h_0\cdot 2^{-l_1}$, $k = k_0\cdot 2^{-2l_2}$.
\end{theorem}
\begin{proof}
See Section \ref{sec_proofofmaintheorem}.
\end{proof}
\begin{remark}
The condition $0<\rho\leq 1/\sqrt{2}$ is the same already required for the mean-square stability in Theorem \ref{thm_1dImplicitStability}. It will be used for the damping of high wave number components (see Section \ref{sec_1dhighwave}).
\end{remark}
\begin{remark}
In Theorem~\ref{thm_1dmainresult}, to make sure the condition \eqref{eq_1d_k/hcondition} holds, $k_0$ is not fixed. Instead, $k_0$ depends on $l_1^\ast$. 
\end{remark}

It follows from Theorem~\ref{thm_1dmainresult} that 
the MIMC increments~\eqref{eq_1dDeltaLinDetails} (for scheme \eqref{eq_1dimplicitdiscrete} and approximation \eqref{eq_1dLossApprox} of $L$) satisfy
\begin{align}
  &E_{(l_1,\,l_2)} = \mathbb{E}[\Delta L_{(l_1,\,l_2)}]\leq C_1^\ast\,h_0^2\,k_0\cdot 2^{-2(l_1+l_2)}, \label{eq_1dassumptions_1}\\
  &V_{(l_1,\,l_2)} = \mathrm{Var}[\Delta L_{(l_1,\,l_2)}]\leq C_2^\ast\,h_0^4\,k_0^2\cdot 2^{-4(l_1+l_2)}, \label{eq_1dassumptions_2}\\
  &W_{(l_1,\,l_2)} \,\leq C_3^\ast\,h_0^{-1}k_0^{-1}\cdot2^{\,l_1+2l_2}, \label{eq_1dassumptions_3}
\end{align}
where $C_1^\ast,C_2^\ast,C_3^\ast$ are positive constants, and the constants will be used in the numerical implementation.

These are the Assumptions 1, 2, 3 in \cite{ref5}. We will follow their approach in Section \ref{sec_1dOptimalIndex} to construct the index set $\mathcal{I}$, choose the number $M_{(l_1,\,l_2)}$ of samples and derive the complexity.

Similar to the proof of Theorem~\ref{thm_1dmainresult}, we have the following.
\begin{remark}\label{rmk_1dfirstorderdifference}
Assume $0<\rho\leq 1/\sqrt{2}$. Then there exists a real number $C>0$, such that for any $l_1^\ast\in\mathbb{N}$, any $h_0>0$, and any $k_0>0$ such that
\[
k_0\leq \frac{T\log_2(\theta^{-1})}{C_0+(3+\beta)(l_1^\ast + \log_2h_0^{-1})}\,,
\]
where $0<\theta<1$, $\beta>0$, $C_0>0$ are the same constants as in Remark \ref{rmk_1dImplicitConvergence}, the following holds:

For any $0\leq l_1 \leq l_1^\ast$, $l_2\geq0$ with $h=h_0\cdot 2^{-l_1}$, $k = k_0\cdot 2^{-2l_2}$, the first order differences of $L_{(l_1,\,l_2)}$ derived from \eqref{eq_1dimplicitdiscrete} and \eqref{eq_1dLossApprox} have the first and second moments
\begin{align*}
\big|\mathbb{E}\left[\Delta_1 L_{(l_1,\,l_2)}\right]\big|\leq C\, h^2,&\qquad
\mathbb{E}\big[\left|\Delta_1 L_{(l_1,\,l_2)}\right|^2\big]\leq C\, h^4,\\
\big|\mathbb{E}\left[\Delta_2 L_{(l_1,\,l_2)}\right]\big|\leq C\, k,\ \;&\qquad
\mathbb{E}\big[\left|\Delta_2 L_{(l_1,\,l_2)}\right|^2\big]\leq C\,k^2,
\end{align*}
using the notation from \eqref{eq_1dfirstorderdifference}.
\end{remark}

In the remainder of this section, we give a proof of Theorem~\ref{thm_1dmainresult}. The idea of the proof follows \cite{carter2007sharperror}.

\subsection{Fourier transform of the solution}\label{sec_1d_fouriertransform}
Define  the Fourier transform pair
\begin{align*}
\widetilde{v}(t,\gamma) &= \int_{-\infty}^\infty v(t,x)\mathrm{e}^{-\mathrm{i}\gamma x}\,\mathrm{d}x,\\
v(t,x)&=\frac{1}{2\pi}\int_{-\infty}^\infty\widetilde{v}(t,\gamma)\mathrm{e}^{\mathrm{i}\gamma x}\,\mathrm{d}\gamma.
\end{align*}
The Fourier transform of \eqref{eq_1dSPDE} yields
\begin{equation}\label{eq_1dfourier}
\mathrm{d}\widetilde{v} = -\widetilde{v}\Big(\mathrm{i}\mu\gamma\,\mathrm{d}t + \frac{1}{2}\gamma^2\,\mathrm{d}t + \mathrm{i}\sqrt{\rho}\gamma\,\mathrm{d}M_t\Big),
\end{equation}
subject to the initial data $\widetilde{v}(0,\gamma) = \mathrm{e}^{-\mathrm{i}\gamma x_0}$ from \eqref{eq_1dDiracInitial}. We assume that $\mu =0$ in the following for simplicity, since the results will not change when a drift term appears. (See Remark 2.3 in \cite{ref1}.)

The solution to \eqref{eq_1dfourier} is then
$$\widetilde{v}(t) = X(t)\mathrm{e}^{-\mathrm{i}\gamma x_0},$$
where
\begin{equation}\label{eq_1dsolXn}
X(t) = \exp\Big(-\frac{1}{2}(1-\rho)\gamma^2t-\mathrm{i}\gamma\sqrt{\rho}M_t\Big).
\end{equation}

For the discretised equation, we can use a discrete-continuous Fourier decomposition
$$V_j^0 = \frac{1}{2\pi h}\int_{-\pi} ^{\pi} \widetilde{V}^0(\gamma)\mathrm{e}^{\mathrm{i}(j-j_0)\gamma}\,\mathrm{d}\gamma,$$
for $j_0$ defined above \eqref{eq_1dDiracInitialApprox}, where 
$$\widetilde{V}^0(\gamma) = h\sum_{j=-\infty}^\infty V_j^0\mathrm{e}^{-\mathrm{i}(j-j_0)\gamma}.$$

Since we approximate the Dirac initial data by \eqref{eq_1dDiracInitialApprox}, it follows that $\widetilde{V}^0(\gamma) = 1.$ Then by linearity of the equation, we have
\begin{equation}\label{eq_1dDiscreteFourierV}
V_j^n = \frac{1}{2\pi h}\int_{-\pi}^{\pi} \widetilde{V}^n(\gamma)\mathrm{e}^{\mathrm{i}(j-j_0)\gamma}\,\mathrm{d}\gamma = \frac{1}{2\pi }\int_{-\pi/h}^{\pi/h} \widetilde{V}^n(\gamma)\mathrm{e}^{\mathrm{i}(j-j_0)\gamma h}\,\mathrm{d}\gamma,
\end{equation}
where we make the ansatz
\begin{equation}\label{eq_1dDiscreteFourierAnsatz}
\widetilde{V}^n(\gamma) = X_n(\gamma)\widetilde{V}^0(\gamma) = X_n(\gamma).
\end{equation}

We can write \eqref{eq_1dimplicitdiscrete} as
\begin{equation}\label{eq_1dimplicitdiscrete2}
\bigg(I + \frac{\mu k}{2h}D_1 -\frac{k}{2h^2}D_2\bigg)V^{n+1} = \bigg(I-\frac{\sqrt{\rho k}Z_n}{2h}D_1 + \frac{\rho k(Z_n^2-1)}{8h^2}D_1^2\bigg)V^n.
\end{equation}
Inserting in \eqref{eq_1dimplicitdiscrete2}, a simple calculation gives
\begin{equation}\label{eq_1dCn1}
X_{n+1}(\gamma) = \frac{1-\mathrm{i}c\sqrt{\rho k}Z_n+a\rho k (Z_n^2-1)}{1-\hat{a}k}X_n(\gamma),
\end{equation}
where
\begin{equation}\label{eq_1d_a_ahat_c}
a = -\frac{\sin^2\gamma h}{2h^2},\qquad
\hat{a} = -\frac{2\sin^2(\gamma h/2)}{h^2},\qquad
c = \frac{\sin\gamma h}{h}.
\end{equation}

Following \cite{higham2000mean,ref1}, we say that the implicit Milstein scheme is mean-square stable, provided
\[
\lim_{n\rightarrow \infty}\mathbb{E}\left[|X_n|^2\right]= 0\quad\mbox{for all $\gamma$.}
\]
Therefore, by stationarity we need
$$\mathbb{E}\big[|X_{n+1}|^2\big] < \mathbb{E}\big[|X_n|^2\big],$$
i.e.,
$$\mathbb{E}\bigg[\left|\frac{1-\mathrm{i}c\sqrt{\rho k}Z_n+a\rho k (Z_n^2-1)}{1-\hat{a}k}\right|^2\bigg]< 1.$$
This gives a condition on the correlation for mean-square stability (see \cite{ref1}), stated previously as Theorem~\ref{thm_1dImplicitStability}.

Now that the stability is ensured, we can approximate the functional of the solutions.

\subsection{Fourier analysis of MIMC estimators (proof of Theorem~\ref{thm_1dmainresult})}\label{sec_proofofmaintheorem}
From~\eqref{eq_1dLossApprox}, \eqref{eq_1dDiscreteFourierV} and \eqref{eq_1dDiscreteFourierAnsatz}, we can give a discrete approximation of the functional $L(v(T,\cdot))$ of the form
\begin{align*}
L_{(l_1,\,l_2)} = 1-h\sum_{j=1}^{\infty} V_{j}^N - \frac{h}{2}V_0^N = 1- \frac{1}{2\pi}\int_{-\pi/h}^{\pi/h} hX_N(\gamma) \bigg(\sum_{j=-j_0+1}^{\infty} \mathrm{e}^{\mathrm{i}jh\gamma}+\frac{1}{2}\mathrm{e}^{-\mathrm{i}j_0h\gamma}\bigg)\,\mathrm{d}\gamma.
\end{align*}

To simplify the notation, we define 
\begin{equation}\label{eq_1dFourierNotationSimplify}
\chi(h,\gamma) = \sum_{j=-j_0+1}^{\infty} \mathrm{e}^{\mathrm{i}jh\gamma}+\frac{1}{2}\mathrm{e}^{-\mathrm{i}j_0h\gamma},
\end{equation}
where $j_0= x_0/h$ as before. Note that $\chi$ is defined in a distributional sense, and it only appears multiplied by the smooth, fast decaying function $X_N$ and in integral form, hence this is well-defined.
Then $L_{(l_1,\,l_2)}$ has the form
\begin{equation}\label{eq_1dloss}
L_{(l_1,\,l_2)} = 1-\frac{1}{2\pi}\int_{-\pi/h}^{\pi/h} hX_N(\gamma) \chi(h,\gamma)\,\mathrm{d}\gamma.
\end{equation}

Now we derive the leading order term of $\Delta L_{(l_1,\,l_2)}$, assuming that $\rho\leq 1/\sqrt{2}$ holds.

Let $N' = N/4$ be the final timestep for level $(l_1,\,l_2-1)$. We write $X_N(\gamma)$ as~$X_N$, and from~\eqref{eq_1dloss} we have
\begin{equation}\label{eq_1dOriginalDeltaL}
\begin{aligned}
\Delta L_{(l_1,\,l_2)}&= L_{(l_1,\,l_2)}-L_{(l_1,\,l_2-1)}-L_{(l_1-1,\,l_2)}+L_{(l_1-1,\,l_2-1)}\\
 &= \frac{1}{2\pi}\int_{|\gamma|<\frac{\pi}{2h}} -h\big(X_N^{l_1,\,l_2}- X_{N'}^{l_1,\,l_2-1}\big)\chi(h,\gamma) + 2h\big(X_N^{l_1-1,\,l_2}- X_{N'}^{l_1-1,\,l_2-1}\big)\chi\big(2h,\gamma\big)\,\mathrm{d}\gamma\\
&\quad - \frac{1}{2\pi}\int_{\frac{\pi}{2h}<|\gamma|<\frac{\pi}{h}}h\big(X_N^{l_1,\,l_2}- X_{N'}^{l_1,\,l_2-1}\big)\chi(h,\gamma)\,\mathrm{d}\gamma.
\end{aligned}
\end{equation}

We introduce $\Delta\widehat{L}_{(l_1,\,l_2)}(\gamma)$ as the integrand in \eqref{eq_1dOriginalDeltaL},
\begin{equation}\label{eq_DeltaL1}
\Delta\widehat{L}_{(l_1,\,l_2)}(\gamma) = \begin{cases}
-h\big(X_N^{l_1,\,l_2}- X_{N'}^{l_1,\,l_2-1}\big)\chi(h,\gamma) + 2h\big(X_N^{l_1-1,\,l_2}- X_{N'}^{l_1-1,\,l_2-1}\big)\chi(2h,\gamma),\quad\mbox{for $|\gamma|<\frac{\pi}{2h}$}\\
-h\big(X_N^{l_1,\,l_2}- X_{N'}^{l_1,\,l_2-1}\big)\chi(h,\gamma),\qquad\mbox{for $\frac{\pi}{2h}<|\gamma|<\frac{\pi}{h}$.}
\end{cases}
\end{equation}
Hence we can express $\Delta L_{(l_1,\,l_2)}$ as
\begin{equation}\label{eq_DeltaL}
\Delta L_{(l_1,\,l_2)} = \frac{1}{2\pi}\int_{|\gamma|<\frac{\pi}{h}}\Delta\widehat{L}_{(l_1,\,l_2)}(\gamma)\,\mathrm{d}\gamma.\\
\end{equation}
Moreover, after some computation we have for $|\gamma|<\frac{\pi}{2h}$,
\begin{equation}\label{eq_DeltaL0}
\begin{aligned}
\Delta\widehat{L}_{(l_1,\,l_2)}(\gamma) &= -2h\cdot \chi(2h,\gamma)\cdot\bigg(X_N^{l_1,\,l_2} - X_{N'}^{l_1,\,l_2-1} - X_N^{l_1-1,\,l_2} + X_{N'}^{l_1-1,\,l_2-1}\bigg)\\
&\quad -\big(h\cdot\chi(h,\gamma)-2h\cdot\chi(2h,\gamma)\big)\cdot\bigg(X_N^{l_1,\,l_2} - X_{N'}^{l_1,\,l_2-1}\bigg).
\end{aligned}
\end{equation}

Following the analysis in \cite{carter2007sharperror}, we compare the numerical solution to the analytical solution by splitting the domain into two wave number regions. Assume $p$ is a constant satisfying $0<p<\frac{1}{4}$. Then we define the low wave number region by
\begin{equation}\label{eq_1d_Omega_low}
\Omega_{\text{low}} = \big\{\gamma: \vert\gamma\vert\leq\min\{h^{-2p},\,k^{-p}\}\big\},
\end{equation}
and the high wave number region by
\begin{equation}\label{eq_1d_Omega_high}
\Omega_{\text{high}} = \big\{\gamma: \vert\gamma\vert>\min\{h^{-2p},\,k^{-p}\}\big\}\cap \left[-\frac{\pi}{h},\frac{\pi}{h}\right].
\end{equation}
Then we have the following lemmas.
\begin{lemma}[Low wave region]\label{lem_1dlowwave}
For $\Delta\widehat{L}_{(l_1,\,l_2)}$ introduced in \eqref{eq_DeltaL1}, there exists a constant~$C>0$, such that for any $l_1,l_2\in\mathbb{N}$,
\[
\int_{\Omega_{\text{low}}}\Delta \widehat{L}_{(l_1,\,l_2)}(\gamma)\,\mathrm{d}\gamma = h^2k\cdot R(T),
\]
where $R(T)$ is a random variable with bounded first and second moments satisfying
\[
\Big|\mathbb{E}\big[R(T)\big]\Big|\leq C,\quad \mathbb{E}\Big[\big|R(T)\big|^2\Big]\leq C.
\]

\end{lemma}
\begin{proof}
See Section \ref{sec_1dlowwaveregion}.
\end{proof}

\begin{lemma}[High wave region]\label{lem_1dhighwave}
There exists a real number $C>0$, such that for any $l_1^\ast\in\mathbb{N}$, any $h_0>0$, and any $k_0>0$ such that
\[
k_0\leq \frac{T\log_2(\theta^{-1})}{C_0+(3+\beta)(l_1^\ast + \log_2h_0^{-1})}\,,
\]
where $0<\theta<1$, $\beta>0$, $C_0>0$ are the same constants as in Remark \ref{rmk_1dImplicitConvergence}, the following holds for $h_0\,2^{-l_1^\ast}\leq h\leq h_0$:
\begin{align*}
\int_{\Omega_{\mathrm{high}}}\Delta \widehat{L}_{(l_1,\,l_2)}(\gamma)\,\mathrm{d}\gamma = h^{2+\frac{1}{2}\beta}\,k^r\cdot\widetilde{R}(T),\quad \forall\ r>0,
\end{align*}
where $\widetilde{R}(T)$ is a random variable with bounded first and second moments satisfying
\[
\Big|\mathbb{E}\big[\widetilde{R}(T)\big]\Big|\leq C,\quad \mathbb{E}\Big[\big|\widetilde{R}(T)\big|^2\Big]\leq C.
\]
\end{lemma}
\begin{proof}
See Section \ref{sec_1dhighwave}.
\end{proof}

Combining Lemma \ref{lem_1dlowwave} and Lemma \ref{lem_1dhighwave}, we get
\[
\Delta L_{(l_1,\,l_2)} =\frac{1}{2\pi}\int_{|\gamma|<\frac{\pi}{h}}\Delta\widehat{L}_{(l_1,\,l_2)}(\gamma)\,\mathrm{d}\gamma = h^2k\cdot R(T) + h^{2+\frac{1}{2}\beta}\,k^r\cdot\widetilde{R}(T).
\]

Hence there exists a constant $C>0$ independent of $(l_1,l_2)$, such that the first and second moments of $\Delta L_{(l_1,\,l_2)}$ satisfy
\[
\big|\mathbb{E}\left[\Delta L_{(l_1,\,l_2)}\right]\big| \leq C\, h^2k,\qquad
\mathbb{E}\big[\left|\Delta L_{(l_1,\,l_2)}\right|^2\big] \leq C\, h^4k^2.
\]

This concludes the proof of Theorem~\ref{thm_1dmainresult}. We will prove Lemma \ref{lem_1dlowwave} and Lemma \ref{lem_1dhighwave} in the remainder of this section.

\subsection{Low wave number region (proof of Lemma \ref{lem_1dlowwave})}\label{sec_1dlowwaveregion}
First consider the case where $\gamma$ is small, such that 
$\gamma\in\Omega_{\text{low}}$. To analyse the mean and the variance of $\Delta\widehat{L}_{(l_1,\,l_2)}(\gamma)$ in \eqref{eq_DeltaL0}, it suffices to analyse the mean and the variance of
\begin{align*}
X_N^{l_1,\,l_2} - X_{N'}^{l_1,\,l_2-1} - X_N^{l_1-1,\,l_2} + X_{N'}^{l_1-1,\,l_2-1}\qquad\text{and}\qquad X_N^{l_1,\,l_2} - X_{N'}^{l_1,\,l_2-1}.
\end{align*}

From \eqref{eq_1dsolXn} it follows that the exact solution of $X(t_{n+1})$ given $X(t_n)$ is
\begin{equation}\label{eq_exactsol_leveln}
X(t_{n+1}) = X(t_n)\exp\Big(-\frac{1}{2}(1-\rho)\gamma^2k-\mathrm{i}\gamma\sqrt{\rho }W_n\Big),
\end{equation}
where $M_{t_{n+1}} - M_{t_n} \coloneqq W_n$ is the Brownian increment.

Now we consider the numerical approximation of \eqref{eq_exactsol_leveln}, written as
$$X_{n+1}^{l_1,\,l_2} = C_n^{l_1,\,l_2}\,X_n^{l_1,\,l_2},$$
where
\begin{equation}\label{eq_1dCn2}
C_n^{l_1,\,l_2} \coloneqq \exp\bigg(-\frac{1}{2}(1-\rho)\gamma^2k-\mathrm{i}\gamma\sqrt{\rho }W_n + e_n^{l_1,l_2}\bigg),
\end{equation}
and $e_n^{l_1,\,l_2}$ is the logarithmic error between the numerical solution and the exact solution introduced during $[nk,(n+1)k]$.
Aggregating over $N$ timesteps, at $t_N = kN = T$, 
$$X_N^{l_1,l_2} = X(T)\exp\bigg(\sum_{n=0}^{N-1}e_n^{l_1,l_2}\bigg),$$
where $X(T) = \exp(-\frac{1}{2}(1-\rho)\gamma^2T-\mathrm{i}\gamma\sqrt{\rho}\,W_T)$ is the exact solution at time $T$. Therefore
\begin{equation}\label{eq_1dX1-X2}
\begin{aligned}
&\ X_N^{l_1,\,l_2}-X_{N'}^{l_1,\,l_2-1}\\
=&\ X(T)\bigg[\exp\Big(\sum_{n=0}^{N-1}e_n^{l_1,l_2}\Big) - \exp\Big(\sum_{n=0}^{N'-1}e_n^{l_1,l_2-1}\Big)\bigg],\\[5pt]
&\ X_N^{l_1,\,l_2} - X_{N'}^{l_1,\,l_2-1} - X_N^{l_1-1,\,l_2} + X_{N'}^{l_1-1,\,l_2-1}\\
=&\ X(T)\bigg[\exp\Big(\sum_{n=0}^{N-1}e_n^{l_1,l_2}\Big)- \exp\Big(\sum_{n=0}^{N'-1}e_n^{l_1,l_2-1}\Big) - \exp\Big(\sum_{n=0}^{N-1}e_n^{l_1-1,l_2}\Big) + \exp\Big(\sum_{n=0}^{N'-1}e_n^{l_1-1,l_2-1}\Big)\bigg].\\
\end{aligned}
\end{equation}

From (\ref{eq_1dCn2}), we have
\begin{equation}
\sum_{n=0}^{N-1}e_n^{l_1,\,l_2} = \sum_{n=0}^{N-1}\log C_n^{l_1,\,l_2} +\frac{1}{2}(1-\rho)\gamma^2T + \mathrm{i}\gamma\sqrt{\rho}\,W_T,
\end{equation}
in which $C_n^{l_1,\,l_2}$ has the same form as in (\ref{eq_1dCn1}), 
$$C_n^{l_1,\,l_2} =\frac{1-\mathrm{i}c\sqrt{\rho k}Z_n+a\rho k (Z_n^2-1)}{1-\hat{a}k},$$
where $a$, $\hat{a}$ and $c$ are as in \eqref{eq_1d_a_ahat_c}.

In the following, for a small parameter $k$, $\eta_1(k) = o(k)$, $\eta_2(k) = O(k)$, $\eta_3(k) = o(1)$, $\eta_4(k) = O(1)$ denote, respectively,
\begin{align*}
\limsup\limits_{k\rightarrow 0}\mathbb{E}\bigg[\bigg(\frac{\eta_1(k)}{k}\bigg)^2\bigg] &= 0,\qquad \limsup\limits_{k\rightarrow 0}\mathbb{E}\bigg[\bigg(\frac{\eta_2(k)}{k}\bigg)^2\bigg]<\infty,\\
\limsup\limits_{k\rightarrow 0}\mathbb{E}\Big[\big( \eta_3(k)\big)^2\Big] &= 0,\qquad \limsup\limits_{k\rightarrow 0}\mathbb{E}\Big[\big( \eta_4(k)\big)^2\Big]<\infty.
\end{align*}

Then one can derive from Taylor expansion of $\log C_n^{l_1,\,l_2}$,
\begin{align*}
\sum_{n=0}^{N-1} e_n^{l_1,\,l_2} =&\ \mathrm{i}(\gamma-c)\sqrt{\rho k}\sum_{n=0}^{N-1}Z_n + \big(a+\frac{1}{2}c^2\big)\rho k\sum_{n=0}^{N-1}Z_n^2 + \big(a+\frac{1}{2}\gamma^2\big)(1-\rho)T + (\hat{a}-a)T\\
& + \mathrm{i}\big(ac+\frac{1}{3}c^3\big)\rho k\sqrt{\rho k}\sum_{n=0}^{N-1}Z_n^3 - \mathrm{i}ac\rho k\sqrt{\rho k}\sum_{n=0}^{N-1}Z_n  - \big(\frac{1}{2}a^2+ac^2+\frac{1}{4}c^4\big)\rho^2k^2\sum_{n=0}^{N-1}Z_n^4\\
& + (a^2+ac^2)\rho^2k^2\sum_{n=0}^{N-1}Z_n^2 + \frac{1}{2}(\hat{a}^2-a^2\rho^2)kT + \gamma^4 o(k),
\end{align*}
where we have 
\begin{align*}
\mathbb{E}\left[\left(\sum_{n=0}^{N-1} Z_n^{2m}\right)^2\right] \leq \frac{(4m-1)!!\,T^2}{k^2},\quad \mathbb{E}\left[\left(\sum_{n=0}^{N-1} Z_n^{2m-1}\right)^2\right] =\frac{(4m-3)!!T}{k}.
\end{align*}
Moreover, this discretisation scheme satisfies
$$ a+\frac{1}{2}c^2 = 0. $$
This is important since the term $(a+\frac{1}{2}c^2)\rho k\sum Z_n^2$ vanishes in this case, and it follows that
\begin{align*}
&\exp\bigg(\sum_{n=0}^{N-1} e_n^{l_1,\,l_2}\bigg) 
= \exp\bigg(\Big(\hat{a}-a+(a+\frac{1}{2}\gamma^2)(1-\rho)\Big)T\bigg)\cdot\bigg[1+\mathrm{i}(\gamma-c)\sqrt{k}\sum_{n=0}^{N-1}Z_n+ \frac{1}{2}(\hat{a}^2\\
&\ -a^2\rho^2)kT- \frac{1}{2}(\gamma-c)^2\rho k\bigg(\sum_{n=0}^{N-1}Z_n\bigg)^2+ \mathrm{i}\gamma^3k\sqrt{\rho k}\sum_{n=0}^{N-1}\phi_1(Z_n) + \gamma^4k^2\sum_{n=0}^{N-1}\phi_2(Z_n)+\gamma^4k\cdot r_0(k)\bigg],
\end{align*}
where $\phi_1(\cdot)$ is an odd degree polynomial function, $\phi_2(\cdot)$ is an even degree polynomial function, and $r_0(k)$ is a random variable with bounded second moments such that
\[
\lim\limits_{l_2\rightarrow \infty}\mathbb{E}\big[|r_0(k)|^2\big]\rightarrow 0.
\]

Note that the level $(l_1,\,l_2-1)$ has mesh size $h$ and timestep $4k$, such that it has $N'=N/4$ time steps. To reduce the variance, we use the same Wiener process in these two levels, hence the Brownian increment $\widetilde{Z}_n$ on level $(l_1,\,l_2-1)$ satisfies
\begin{align*}
\widetilde{Z}_n &= \frac{1}{2}(Z_{4n-3}+Z_{4n-2}+Z_{4n-1}+Z_{4n}),\quad\text{thus }\ \sum_{n=0}^{N-1}Z_n = 2\sum_{n=0}^{N'-1}\widetilde{Z}_n.
\end{align*}

As a result, we have
\begin{equation}\label{eq_1d2ExpDifference}
\begin{aligned}
&\exp\bigg(\sum_{n=0}^{N-1} e_n^{l_1,\,l_2}\bigg) - \exp\bigg(\sum_{n=0}^{N'-1} e_n^{l_1,\,l_2-1}\bigg) \\
=&\ \exp\bigg(\Big(\hat{a}-a+(a+\frac{1}{2}\gamma^2)(1-\rho)\Big)T\bigg)\cdot\bigg[ -\frac{3}{2}(\hat{a}^2-a^2\rho^2)kT +\mathrm{i}\gamma^3 k\sqrt{k}\bigg(\sum_{n=0}^{N-1}\phi_1(Z_n)\quad\\
&\quad - \sum_{n=0}^{N'-1}\phi_1(\widetilde{Z}_n)\bigg)  + \gamma^4k^2\bigg(\sum_{n=0}^{N-1}\phi_2(Z_n) - \sum_{n=0}^{N'-1}\phi_2(\widetilde{Z}_n)\bigg) + \gamma^4k\cdot r_1(k)\bigg]\\
:=&\ k\cdot E_1(T,\gamma),
\end{aligned}
\end{equation}
where 
\[
\lim\limits_{l_2\rightarrow \infty}\mathbb{E}\big[|r_1(k)|^2\big]\rightarrow 0,
\]
\begin{align*}
E_1(T,\gamma) &= \exp\bigg(\Big(\hat{a}-a+(a+\frac{1}{2}\gamma^2)(1-\rho)\Big)T\bigg)\cdot\bigg[ -\frac{3}{2}(\hat{a}^2-a^2\rho^2)T +\mathrm{i}\gamma^3 \sqrt{k}\bigg(\sum_{n=0}^{N-1}\phi_1(Z_n)\quad\\
&\quad - \sum_{n=0}^{N'-1}\phi_1(\widetilde{Z}_n)\bigg)  + \gamma^4k\bigg(\sum_{n=0}^{N-1}\phi_2(Z_n) - \sum_{n=0}^{N'-1}\phi_2(\widetilde{Z}_n)\bigg) + \gamma^4\,r_1(k)\bigg].
\end{align*}

Since
\begin{align*}
a &= -\frac{\gamma^2}{2}\cdot\frac{\sin^2 h\gamma}{h^2\gamma^2} = -\frac{\gamma^2}{2} + \frac{\gamma^4}{6}h^2 + \gamma^6O(h^4),\\
\hat{a} &= -\frac{\gamma^2}{2}\cdot\frac{\sin^2 (h\gamma/2)}{(h\gamma/2)^2} = -\frac{\gamma^2}{2} + \frac{\gamma^4}{24}h^2 + \gamma^6O(h^4),\\
c &= \gamma\cdot\frac{\sin h\gamma}{h\gamma} = \gamma - \frac{\gamma^3}{6}h^2 + \gamma^5O(h^4),
\end{align*}
it follows that
$$ \exp\left(\Big(\hat{a}-a+(a+\frac{1}{2}\gamma^2)(1-\rho)\Big)T\right)= 1+\frac{1-4\rho}{24}T\cdot \gamma^4h^2 + \gamma^6O(h^4).$$
Therefore we derive
\begin{equation}\label{eq_1d4ExpDifference}
\begin{aligned}
&\exp\bigg(\sum_{n=0}^{N-1} e_n^{l_1,\,l_2}\bigg) - \exp\bigg(\sum_{n=0}^{N'-1} e_n^{l_1,\,l_2-1}\bigg) - \exp\bigg(\sum_{n=0}^{N-1} e_n^{l_1-1,\,l_2}\bigg) + \exp\bigg(\sum_{n=0}^{N'-1} e_n^{l_1-1,\,l_2-1}\bigg) \\[3pt]
=&\ \mathrm{i} \gamma^7k\sqrt{k}\bigg(\sum_{n=0}^{N-1}\phi_1(Z_n) - \sum_{n=0}^{N'-1}\phi_1(\widetilde{Z}_n)\bigg)\cdot O(h^2) + \gamma^8 k^2\bigg(\sum_{n=0}^{N-1}\phi_2(Z_n)\\
&\quad - \sum_{n=0}^{N'-1}\phi_2(\widetilde{Z}_n)\bigg)\cdot O(h^2) + \gamma^8 h^2k\frac{(1-\rho^2)(4\rho-1)T^2}{64} + \gamma^8h^2k\cdot r_2(k)\\
\coloneqq&\ h^2k\cdot E_2(T,\gamma),
\end{aligned}
\end{equation}
where
\[
\lim\limits_{l_2\rightarrow \infty}\mathbb{E}\big[|r_2(k)|^2\big]\rightarrow 0,
\]
\begin{align*}
E_2(T,\gamma) &=  \mathrm{i} \gamma^7\sqrt{k}\bigg(\sum_{n=0}^{N-1}\phi_1(Z_n) - \sum_{n=0}^{N'-1}\phi_1(\widetilde{Z}_n)\bigg)\cdot O(1) + \gamma^8 k\bigg(\sum_{n=0}^{N-1}\phi_2(Z_n)\\
&\quad - \sum_{n=0}^{N'-1}\phi_2(\widetilde{Z}_n)\bigg)\cdot O(1) + \gamma^8 \frac{(1-\rho^2)(4\rho-1)T^2}{64} + \gamma^8\cdot r_2(k). 
\end{align*}
So we have
\begin{align*}
X_N^{l_1,\,l_2} - X_{N'}^{l_1,\,l_2-1} &= k\cdot X(T)\cdot E_1(T,\gamma),\\
X_N^{l_1,\,l_2} - X_{N'}^{l_1,\,l_2-1} - X_N^{l_1-1,\,l_2} + X_{N'}^{l_1-1,\,l_2-1} &= h^2k\cdot X(T)\cdot E_2(T,\gamma).
\end{align*}

Here, for $i=1,2$,
\[
\int_{\Omega_{\text{low}}} h\,\chi(h,\gamma)\cdot X(T)\,E_i(T,\gamma)\,\mathrm{d}\gamma
\]
is the numerical approximation to 
\[
\int_0^\infty\int_{\Omega_\text{low}}\mathrm{e}^{\mathrm{i}\gamma(x-x_0)}\cdot X(T)\, E_i(T,\gamma)\,\mathrm{d}\gamma \,\mathrm{d}x.
\]
Because of the exponential decay of $X(T)$ in $\gamma$, there exists $C_0$ independent of $(l_1,l_2)$, such that
\[
\mathbb{E}\Bigg|\int_{\Omega_{\text{low}}} h\cdot\chi(h,\gamma)\Big(X(T)\cdot E_i(T,\gamma)\Big)\,\mathrm{d}\gamma - \int_0^\infty\int_{\Omega_\text{low}}\mathrm{e}^{\mathrm{i}\gamma(x-x_0)}X(T)\cdot E_i(T,\gamma)\,\mathrm{d}\gamma \,\mathrm{d}x \Bigg|\leq C_0\,h^2.
\]
Then there exists a constant $C_1>0$ independent of $l_1$ and $l_2$, such that
\begin{align*}
& \int_{\Omega_{\text{low}}} \big(h\cdot\chi(h,\gamma)-2h\cdot\chi(2h,\gamma)\big)\cdot \big( X_N^{l_1,\,l_2} - X_{N'}^{l_1,\,l_2-1} \big)\,\mathrm{d}\gamma\\
=\ & h^2k \int_{\Omega_{\text{low}}} \big(h\cdot\chi(h,\gamma)-2h\cdot\chi(2h,\gamma)\big)\cdot X(T)\,E_1(T,\gamma)\,\mathrm{d}\gamma \coloneqq h^2k\cdot R_1(T),\\[5pt]
& \int_{\Omega_{\text{low}}} 2h\,\chi(2h,\gamma)\cdot\big(X_N^{l_1,\,l_2} - X_{N'}^{l_1,\,l_2-1} - X_N^{l_1-1,\,l_2} + X_{N'}^{l_1-1,\,l_2-1}\big)\,\mathrm{d}\gamma\\
=\ & h^2k \int_{\Omega_{\text{low}}} 2h\,\chi(2h,\gamma)\cdot X(T)\,E_2(T,\gamma)\,\mathrm{d}\gamma \coloneqq h^2k\cdot R_2(T),
\end{align*}
where $R_1(T),\,R_2(T)$ are random variables with bounded moments satisfying
\[
\Big|\mathbb{E}\big[R_i(T)\big]\Big|\leq C_1,\quad \mathbb{E}\Big[\big|R_i(T)\big|^2\Big]\leq C_1,\qquad i=1,2.
\]
Then Lemma~\ref{lem_1dlowwave} follows by letting
\[
R(T) = -\big(R_1(T) + R_2(T)\big),\quad C = 4C_1.
\]

\subsection{High wave number region (proof of Lemma \ref{lem_1dhighwave})}\label{sec_1dhighwave}
Now we consider the case where $\gamma$ is large, such that $\gamma\in\Omega_{\text{high}}$. By \eqref{eq_1dCn1}, we have 
\begin{align*}
X_N &= X_0\prod_{n=0}^{N-1}\frac{1-\mathrm{i}\sqrt{k}cZ_n+k\rho a(Z_n^2-1)}{1-k\hat{a}}\\
&= X_0\prod_{n=0}^{N-1}\bigg(1-\mathrm{i}\gamma\sqrt{\rho k}Z_n\frac{\sin h\gamma}{h\gamma}-\frac{\rho}{2}\gamma^2k(Z_n^2-1)\frac{\sin^2 h\gamma}{(h\gamma)^2}\bigg)\bigg(1+\frac{1}{2}\gamma^2k\frac{\sin^2 (h\gamma/2)}{(h\gamma/2)^2}\bigg)^{-1}\\
&= X_0\prod_{n=0}^{N-1}\bigg(1-\mathrm{i}\gamma\sqrt{\rho k} Z_n\frac{\sin h\gamma }{h\gamma}-\frac{\rho}{2}\gamma^2ku(Z_n^2-1)\cos^2\frac{h\gamma}{2}\bigg)\bigg(1+\frac{1}{2}\gamma^2ku\bigg)^{-1},
\end{align*}
denoting
$$u = \frac{\sin^2\frac{h\gamma}{2}}{(\frac{h\gamma}{2})^2}\Longrightarrow \frac{\sin^2h\gamma}{(h\gamma)^2}=\cos^2\frac{h\gamma}{2}\cdot u.$$
Note that $u\in [4/\pi^2,1]$ when $|h\gamma|<\pi$. We have
\begin{equation}
\lim_{N\rightarrow\infty}\mathbb{E}\big[X_N\big] = \lim_{N\rightarrow\infty}X_0\bigg(1+\frac{1}{2}\gamma^2ku\bigg)^{-N} = X_0\cdot \exp\bigg(-\frac{1}{2}\gamma^2uT\bigg),
\end{equation}
and
\begin{align*}
\lim_{N\rightarrow\infty}\mathbb{E}\big[|X_N|^2\big] &= \lim_{N\rightarrow\infty}X_0^2\bigg(\frac{1+\rho\gamma^2ku\cos^2\frac{h\gamma}{2}+\frac{\rho^2}{2}\gamma^4k^2u^2\cos^4\frac{h\gamma}{2}}{1+\gamma^2ku+\frac{1}{4}\gamma^4k^2u^2}\bigg)^N,\\
&= \lim_{N\rightarrow\infty}X_0^2\bigg(1-\gamma^2ku\frac{1-\rho\cos^2\frac{h\gamma}{2}+\frac{1}{4}(1-2\rho^2\cos^4\frac{h\gamma}{2})\gamma^2ku}{1+\gamma^2ku + \frac{1}{4}\gamma^4k^2u^2}\bigg)^N\\
&= X_0^2\cdot\exp\bigg(-\gamma^2uT\frac{1-\rho\cos^2\frac{h\gamma}{2}+\frac{1}{4}(1-2\rho^2\cos^4\frac{h\gamma}{2})\gamma^2ku}{1+\gamma^2ku + \frac{1}{4}\gamma^4k^2u^2}\bigg).
\end{align*}

Since we have assumed that $2\rho^2\leq1$, the numerator~$1-\rho\cos^2\frac{h\gamma}{2}+\frac{1}{4}(1-2\rho^2\cos^4\frac{h\gamma}{2})\gamma^2ku$ is positive.

When $h^2\geq k$, the high wave region is
\[
h^{-2p} < |\gamma| < \pi h^{-1}.
\]

In this case, 
\[
\mathbb{E}\bigg[\,\bigg|\int_{\Omega_{\mathrm{high}}}\Delta \widehat{L}_{(l_1,\,l_2)}(\gamma)\,\mathrm{d}\gamma \bigg|^2\,\bigg]= o(h^r),\qquad \forall r>0.
\]

When $h^2<k$, following the proof of Proposition \ref{prop_1dImplicitConvergence} in Appendix \ref{appendix_proof_prop_meansquareconvergence}, we have
\begin{equation}\label{eq_1d_highintergral}
\int_{\Omega_{\text{high}}} \mathbb{E}\left|X_N(\gamma)\right|^2\,\mathrm{d}\gamma \leq \widetilde{C}\,\theta^{N}h^{-1},
\end{equation}
where $\widetilde{C}>0,\,0<\theta<1$ are constants independent of $h$ and $k$. As $k_0$ satisfies the condition 
\[
k_0\leq \frac{T\log_2(\theta^{-1})}{C_0+(3+\beta)(l_1^\ast + \log_2h_0^{-1})}\,,
\]
or, equivalently,
\[
\theta^{\frac{T}{k_0}}\leq 2^{-C_0}\cdot h^{3+\beta},\qquad\mbox{for all $h\geq h_0\,2^{-l_1^\ast}$,}
\]
by using \eqref{eq_DeltaL0} and a similar argument as in Appendix~\ref{appendix_proof_prop_meansquareconvergence},
it follows that there exists $C>0$ independent of $h$ and $k$, such that for all $r>0$,
\begin{align*}
\bigg|\mathbb{E}\bigg[\int_{\Omega_{\mathrm{high}}}\Delta \widehat{L}_{(l_1,\,l_2)}(\gamma)\,\mathrm{d}\gamma \bigg]\bigg| \leq C\cdot h^{2+\frac{1}{2}\beta}k^r,\quad \mathbb{E}\bigg[\bigg|\int_{\Omega_{\mathrm{high}}}\Delta \widehat{L}_{(l_1,\,l_2)}(\gamma)\,\mathrm{d}\gamma \bigg|^2\bigg]  \leq C\cdot h^{2+\frac{1}{2}\beta}k^r,
\end{align*}
since in \eqref{eq_1d_highintergral},
\[
\theta^Nh^{-1}=\theta^{N_0}h^{-1}\theta^{\frac{k_0}{k}}\leq C\cdot h^{2+\frac{1}{2}\beta}\theta^{\frac{k_0}{k}}.
\]

This concludes the proof of Lemma \ref{lem_1dhighwave}.

\section{Optimal index set and complexity}\label{sec_1dOptimalIndex}
So far, we have derived the first and second moments of $\Delta L_{(l_1,\,l_2)}$. Based on these, we now need to find an appropriate index set $\mathcal{I}$ for the MIMC estimator \eqref{eq_1dMIMCestimator}.
We first do this for scheme \eqref{eq_1dimplicitdiscrete} and approximation \eqref{eq_1dLossApprox}.

Finding the optimal (with respect to the computational cost required for a given prescribed accuracy $\alpha\varepsilon$) index set $\mathcal{I}$ is equivalent to solving the optimisation problem
\begin{equation}
\min_{\mathcal{I}\subset\mathbb{N}^2}W(\mathcal{I})\qquad \text{such that}\quad \Big|\mathbb{E}\big[\widehat{L}_{\mathcal{I}}-L\big]\Big|<\alpha\varepsilon,
\end{equation}
where $W$ is the total work as in \eqref{eq_1dTotalwork} and $\widehat{L}_{\mathcal{I}}$ the MIMC estimator for $L$ as in \eqref{eq_1dMIMCestimator}, $\alpha\in(0,1)$. From \eqref{eq_1d_weakerror}, we have
\[
\Bigg|\sum_{(l_1,\,l_2)\in\mathcal{I}}\mathbb{E}\big[\Delta L_{(l_1,\,l_2)}\big] -\mathbb{E}\big[L\big]\Bigg|\leq\varepsilon_0+\varepsilon_1,
\]
where $\varepsilon_0$ defined in \eqref{eq_1d_weakerror0} is the error between $L_{(l^\ast,l^\ast)}$ and $L$, and $\varepsilon_1$ defined in \eqref{eq_1d_weakerror1} is the error between two index sets. We let $\varepsilon_0 = (\alpha\varepsilon)^{1+r}$ with $r>0$, $0<\alpha<1$, and $\varepsilon_1 = \alpha\varepsilon$. Then the weak error between $L_{(l_1^\ast,l_2^\ast)}$ and $L$ has a higher order, and the dominant weak error of the MIMC estimator would be $\varepsilon_1 = \alpha\varepsilon$.

Firstly we find the $l_1^\ast,\,l_2^\ast$ and $k_0$. From Proposition \ref{prop_1dErrorofLoss}, we have
\[
\sqrt{\mathbb{E}\big[ |L_{(l_1,\,l_2)} - L|^2 \big]}\leq C(h^2+k).
\]

 For
\[
h_{l_1} = h_0\cdot2^{-l_1},\quad k_{l_2} = k_0\cdot 2^{-2l_2},
\]
it is sufficient to have
\begin{equation}\label{eq_cond1_l_1_l_2}
\begin{aligned}
C\,h_{l_1^\ast}^2 = C h_0^2\cdot2^{-2l_1^\ast} &\leq (\alpha\varepsilon)^{1+r},\\
C\,k_{l_2^\ast} =C\,k_0\cdot 2^{-2l_2^\ast} &\leq (\alpha\varepsilon)^{1+r}\\
\theta^{T/k_0} &= \varepsilon^{2(1+r)}.
\end{aligned}
\end{equation}

Hence we have
\begin{equation}\label{eq_cond2_l_1_l_2}
\begin{aligned}
l_1^\ast &\geq \frac{1+r}{2}\log_2(\varepsilon^{-1}) +\frac{1+r}{2}\log_2(\alpha^{-1})+ \frac{1}{2}\log_2(C\,h_0^2) = O(\log_2(\varepsilon^{-1})),\\
l_2^\ast &\geq \frac{1+r}{2}\log_2(\varepsilon^{-1}) +\frac{1+r}{2}\log_2(\alpha^{-1})+ \frac{1}{2}\log_2(C\,k_0) = O(\log_2(\varepsilon^{-1})), \\
k_0 &= \frac{T\log_2(\theta^{-1})}{2(1+r)\log_2(\varepsilon^{-1})} = O\bigg(\frac{1}{\log(\varepsilon^{-1})}\bigg).
\end{aligned}
\end{equation}

We follow the idea in \cite{ref5}, and introduce the term ``profit'' as
\begin{equation}\label{eq_profit}
P_{\bm{l}} =\frac{E_{\bm{l}}}{\sqrt{V_{\bm{l}}W_{\bm{l}}}}.
\end{equation}
Here $E_{\bm{l}}$, defined as the expectation of $\Delta L_{\bm{l}}$ as before, can be regarded as the contribution to the solution (if included) or error (if not included) originating from level $(l_1,\,l_2)$. The denominator $\sqrt{V_{\bm{l}}W_{\bm{l}}}$ is proportional to the total work on level $(l_1,\,l_2)$, which is given by
\[
	M_{\bm{l}}W_{\bm{l}} = (1-\alpha^2)^{-2}\varepsilon^{-2}\bigg(\sum_{\bm{l}\in\mathcal{I}}\sqrt{V_{\bm{l}}W_{\bm{l}}}\bigg)\sqrt{V_{\bm{l}}W_{\bm{l}}}
\]
from \eqref{eq_M}. This can be regarded as a knapsack problem, where we try to include those levels with small work and large bias. This gives rise to the following heuristic optimisation.

Following Lemma 2.1 in \cite{ref5}, we define a candidate index set as 
$$\{\bm{l}\in \mathbb{N}^2: P_{\bm{l}}>\nu\}\qquad\text{for some }\nu.$$
Note that this class may not contain the optimal index set, but we will see that it is sufficiently rich to obtain optimal complexity order. From \eqref{eq_1dassumptions_1} to \eqref{eq_1dassumptions_3}, we deduce
\begin{equation}
P_{\bm{l}}=\frac{E_{\bm{l}}}{\sqrt{V_{\bm{l}}W_{\bm{l}}}} \approx C_P\cdot 2^{-(w_1l_1+w_2l_2)},
\end{equation}
for some constant $C_P>0$, where
$$w_1 = \frac{1}{2},\qquad w_2 = 1.$$
We introduce strictly positive normalized weights as 
\begin{align*}
\delta_1 = \frac{w_1}{w_1+w_2},\qquad\delta_2 = \frac{w_2}{w_1+w_2}.
\end{align*}
Then $\delta_1+\delta_2 = 1$.
$$\delta_1 = \frac{1}{3},\qquad \delta_2 = \frac{2}{3}.$$
Now the index set is denoted as
\begin{equation}\label{eq_1dindexset_case1}
\mathcal{I}(l^{\ast}) = \big\{(l_1,\,l_2)\in\mathbb{N}^2:\delta_1l_1+\delta_2l_2\leq l^{\ast}\big\}\cap \mathcal{I}_0,
\end{equation}
where $l^{\ast}$ can be derived from the bias constraint
\begin{equation}\label{eq_1dbiasconstraint}
\bigg|\sum_{(l_1,\,l_2)\in\mathcal{I}_0\setminus\mathcal{I}}E_{(l_1,\,l_2)}\bigg|\leq\alpha\varepsilon.
\end{equation}
We try to find the minimal $l^{\ast}$ such that this bias constraint holds. See Figure \ref{fig_1dindexset}.
\begin{figure}[H]
\centering
\includegraphics[width=2.5in, height=2in]{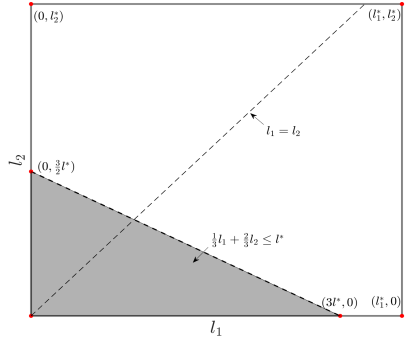}
\caption{Optimal index set.}
\label{fig_1dindexset}
\end{figure}
To this end, we approximate $E_{(l_1,\,l_2)}$ by $C_1^\ast\,h_0^2\,k_0\cdot2^{-2l_1-2l_2}$ from \eqref{eq_1dassumptions_1}. Inserting this in \eqref{eq_1dbiasconstraint}, we choose $l^{\ast}$ such that
\begin{equation}
2^{-3l^{\ast}} = \frac{3\,\alpha}{8C_1^\ast\,h_0^2\,k_0}\varepsilon,\quad \Rightarrow l^{\ast} = \frac{1}{3}\log_2\left(\frac{8C_1^\ast\,h_0^2\,k_0}{3\,\alpha}\varepsilon^{-1}\right).
\end{equation}

Now we compute the computational cost under this index set. Recall from \eqref{eq_totalwork} that
$$W\leq (1-\alpha^2)^{-2}\varepsilon^{-2}\bigg(\sum_{\bm{l}\in\mathcal{I}}\sqrt{V_{\bm{l}}W_{\bm{l}}}\bigg)^2 + \sum_{\bm{l}\in\mathcal{I}}W_{\bm{l}}.$$
Since the first term dominates the second one, we only need to compute
\begin{align*}
\bigg(\sum_{\bm{l}\in\mathcal{I}}\sqrt{V_{\bm{l}}W_{\bm{l}}}\bigg)^2 \sim k_0^{-1}\bigg(\sum_{l_1=0}^{3l^{\ast}}\sum_{l_2=0}^{(-l_1+3l^{\ast})/2}2^{-\frac{3}{2}l_1-l_2}\bigg)^2=\ C\cdot k_0^{-1} = O\big(\big|\log\varepsilon\big|\big),
\end{align*}
where $C$ is a positive constant. Therefore, using discretisation \eqref{eq_1dimplicitdiscrete} and $L_{(l_1,\,l_2)}$ \eqref{eq_1dLossApprox}, the order of the total computational cost $W$ is
\begin{equation}
W= O\left(\varepsilon^{-2}\big|\log\varepsilon\big|\right).
\end{equation}
Although theoretically there is a log term in the complexity, it usually does not affect the numerical tests, as $\theta$ is small, and $\theta^N$ dominates $h^{-1}$.

\section{Some sub-optimal approximations}\label{sec_1dOtherResults}
\subsection{The alternative discretisation \eqref{eq_1dimplicitdiscrete_2}}\label{sec_1dalternative}
Recall the alternative discretisation scheme given in (\ref{eq_1dimplicitdiscrete_2}), in which we apply the Milstein scheme first, then apply a second finite difference with step size $h$. This makes no difference in the convergence order of the multilevel scheme, of which the second moment of $\Delta L_l\coloneqq L_{(l,\,l)} - L_{(l-1,\,l-1)}$ satisfies
\begin{equation}\label{eq_1dMLMCvar}
\mathbb{E}\big[\left|\Delta L_l\right|^2\big]= O(h^4) + O(k^2) + O\left(h^4k\right),
\end{equation}
where $h = h_0\cdot 2^{-l}$, $k = k_0\cdot 4^{-l}.$
As $O(h^2)= O(k)$ in order to balance the bias, $O(h^4k)$ is a higher order term in \eqref{eq_1dMLMCvar}.
However, in the multi-index scheme, $O(h^4k)$ is no longer a higher order term due to the independence of $h$ and $k$. This leads to the difference in complexity between the MLMC and MIMC methods.
\begin{proposition}
Consider the approximation~\eqref{eq_1dimplicitdiscrete_2} and \eqref{eq_1dLossApprox}. Assume $\rho\leq 1/\sqrt{2}$. 
Then there exists a real number $C>0$, such that for any $l_1^\ast\in\mathbb{N}$, any $h_0>0$, and any $k_0>0$ such that
\[
k_0\leq \frac{T\log_2(\theta^{-1})}{C_0+(3+\beta)(l_1^\ast + \log_2h_0^{-1})}\,,
\]
where $0<\theta<1$, $\beta>0$, $C_0>0$ are the same constants as in Remark \ref{rmk_1dImplicitConvergence}, the following holds:

For any $0\leq l_1 \leq l_1^\ast$, $l_2\geq0$ with $h=h_0\cdot 2^{-l_1}$, $k = k_0\cdot 2^{-2l_2}$, the first and second moments of $\Delta L_{(l_1,\,l_2)}$ satisfy
\begin{equation}
\mathbb{E}\left[\Delta L_{(l_1,\,l_2)}\right]\leq C\,h^2k,\quad\mathbb{E}\big[\left|\Delta L_{(l_1,\,l_2)}\right|^2\big] \leq C \,h^4k.
\end{equation}
\end{proposition}

\begin{proof}
In this discretisation scheme,
$$\bigg(I + \frac{\mu k}{2h}D_1 -\frac{k}{2h^2}D_2\bigg)V^{n+1} = \bigg(I-\frac{\sqrt{\rho k}Z_n}{2h}D_1 + \frac{\rho k(Z_n^2-1)}{2h^2}D_2\bigg)V^n.$$
Hence we have
\begin{equation}
X_{n+1}(\gamma) = \frac{1-\mathrm{i}c\sqrt{\rho k}Z_n+\hat{a}\rho k (Z_n^2-1)}{1-\hat{a}k}X_n(\gamma),
\end{equation}
where
\begin{align*}
\hat{a} = -\frac{2\sin^2(\gamma h/2)}{h^2},\qquad
c = \frac{\sin\gamma h}{h}.
\end{align*}
Following the previous analysis, in the low wave region (see Section \ref{sec_1dlowwaveregion}) we can derive that 
\begin{align*}
\sum_{n=0}^{N-1} e_n^{l_1,\,l_2} =& \sum_{n=0}^{N-1}\log C_n^{l_1,\,l_2} +\frac{1}{2}(1-\rho)\gamma^2T + \mathrm{i}\sqrt{\rho}\gamma W_T\\
=&\ \mathrm{i}(\gamma-c)\sqrt{\rho k}\sum_{n=0}^{N-1}Z_n + \Big(\hat{a}+\frac{1}{2}c^2\Big)\rho k\sum_{n=0}^{N-1}Z_n^2 + \Big(\hat{a}+\frac{1}{2}\gamma^2\Big)(1-\rho)T\\
& + \mathrm{i}\Big(\hat{a}c+\frac{1}{3}c^3\Big)\rho k\sqrt{\rho k}\sum_{n=0}^{N-1}Z_n^3 - \mathrm{i}\hat{a}c\rho k\sqrt{\rho k}\sum_{n=0}^{N-1}Z_n  - \Big(\frac{1}{2}\hat{a}^2+\hat{a}c^2+\frac{1}{4}c^4\Big)\rho^2k^2\sum_{n=0}^{N-1}Z_n^4\\
& + \big(\hat{a}^2+\hat{a}c^2\big)\rho^2k^2\sum_{n=0}^{N-1}Z_n^2 + \gamma^4o(k).
\end{align*}
In this case, $\hat{a}+\frac{1}{2}c^2$ is no longer zero, so that
$$ \exp\bigg(\sum_{n=0}^{N-1} e_n^{l_1,\,l_2}\bigg) = \mathrm{e}^{(\hat{a}+\frac{1}{2}\gamma^2)(1-\rho)T}\cdot\bigg(\mathrm{i}(\gamma-c)\sqrt{\rho k}\sum_{n=0}^{N-1}Z_n + \Big(\hat{a}+\frac{1}{2}c^2\Big)\rho k\sum_{n=0}^{N-1}Z_n^2 + r(k) \bigg),$$
where $r(k)$ satisfies
$$\mathbb{E}[r(k)]= O(k),\qquad \mathbb{E}\big[\big|r(k)\big|^2\big] = O(k^2).$$
Now we have
\begin{align*}
\exp\bigg(\sum_{n=0}^{N-1} e_n^{l_1,\,l_2}\bigg) - \exp\bigg(\sum_{n=0}^{N'-1} e_n^{l_1,\,l_2-1}\bigg) =\mathrm{e}^{(\hat{a}+\frac{1}{2}\gamma^2)(1-\rho)T}\cdot\bigg(\big(\hat{a}+\frac{1}{2}c^2\big)\rho k\bigg(\sum_{n=0}^{N-1}Z_n^2-4\sum_{n=0}^{N'-1}\widetilde{Z}_n^2\bigg) + r(k)\bigg).
\end{align*}
However,
\[
\mathbb{E}\bigg(\sum_{n=0}^{N-1}Z_n^2-4\sum_{n=0}^{N'-1}\widetilde{Z}_n^2\bigg)^2 = C_0\cdot k^{-1},
\]
for some constant $C_0$.
Hence similar to the proof of Theorem \ref{thm_1dmainresult}, there exists a constant $C$, such that 
\begin{align*}
\left|\mathbb{E}\left[\int_{\Omega_{\text{low}}}\Delta \widehat{L}_{(l_1,\,l_2)}(\gamma)\,\mathrm{d}\gamma\right]\right| \leq C\,h^2 k,\qquad \mathbb{E}\left[\left|\int_{\Omega_{\text{low}}}\Delta \widehat{L}_{(l_1,\,l_2)}(\gamma)\,\mathrm{d}\gamma\right|^2\right] \leq C\,h^4 k.
\end{align*}
For $\gamma$ in the high wave region, the proof is the same as in Theorem \ref{thm_1dmainresult}.
\end{proof}

By a similar computation to Section \ref{sec_1dOptimalIndex}, the index set is now
$$\mathcal{I}(l^{\ast}) = \left\{(l_1,\,l_2)\in\mathbb{N}^2:\frac{1}{5}l_1+\frac{4}{5}l_2\leq l^{\ast}\right\}\cap\mathcal{I}_0,$$
and the total computational cost $W$ satisfies
$$W= O\left(\varepsilon^{-2}\big|\log\varepsilon\big|^3\right).$$

The corresponding numerical results are demonstrated at the end of Section~\ref{sec_5.1}.

\subsection{An alternative functional approximation}
Instead of approximating the loss by~\eqref{eq_1dLossApprox}, consider a simpler approximation,
\begin{equation}\label{eq_1dcoarseloss}
L_{(l_1,\,l_2)} = 1-h\sum_{j=0}^{\infty} V_{j}^N = 1-\frac{1}{2\pi}\int_{-\pi/h}^{\pi/h} hX_N(\gamma) \bigg(\sum_{j=-j_0}^{\infty} \mathrm{e}^{\mathrm{i}jh\gamma}\bigg)\,\mathrm{d}\gamma,
\end{equation}
where $h = h_0\cdot 2^{-l_1}, k = k_0\cdot 2^{-2l_2}$.
Using the discretisation from the previous subsection, we have the following results.
\begin{proposition}
Consider the approximation~\eqref{eq_1dimplicitdiscrete_2} and \eqref{eq_1dcoarseloss}. Assume $\rho\leq 1/\sqrt{2}$. 
Then there exists a real number $C>0$, such that for any $l_1^\ast\in\mathbb{N}$, any $h_0>0$, and any $k_0>0$ such that
\[
k_0\leq \frac{T\log_2(\theta^{-1})}{C_0+(3+\beta)(l_1^\ast + \log_2h_0^{-1})}\,,
\]
where $0<\theta<1$, $\beta>0$, $C_0>0$ are the same constants as in Remark \ref{rmk_1dImplicitConvergence}, the following holds:

For any $0\leq l_1 \leq l_1^\ast$, $l_2\geq0$ with $h=h_0\cdot 2^{-l_1}$, $k = k_0\cdot 2^{-2l_2}$, the first and second moments of $\Delta L_{(l_1,\,l_2)}$ satisfy
\begin{equation}
\mathbb{E}\left[\Delta L_{(l_1,\,l_2)}\right]\leq C\,hk,\quad\mathbb{E}\big[\left|\Delta L_{(l_1,\,l_2)}\right|^2\big]\leq C\left(h^4k + h^2k^2\right).
\end{equation}
\end{proposition}
\begin{proof}
Similar to the proof of Theorem~\ref{thm_1dmainresult}.
\end{proof}

By the choice of $h$ and $k$ above, we get
$$ O(h^4k)= 2^{-2(l_1+l_2)-2l_1},\quad O(h^2k^2)= 2^{-2(l_1+l_2)-2l_2},$$
which gives
$$O\left(h^4k\right) + O\left(h^2k^2\right)= 2^{-2(l_1+l_2) -2 \min\{l_1,\,l_2\}}.$$
As a consequence, the convergence rate varies between $-2$ and $-3$, depending on which refinement path is chosen.

The peculiarity of this scheme is that it has two leading terms of different orders in the variance, such that the assumptions of the framework in \cite{ref5} (in this case, Assumption 2) are not satisfied. However, the fundamental principles of the MIMC concept still apply and we can split the domain into two parts: $\{(l_1,l_2): l_1\leq l_2\}\cap\mathcal{I}_0$ and $\{(l_1,l_2): l_1> l_2\}\cap\mathcal{I}_0$, and the optimal index set is as shown in Figure~\ref{fig_1d_optimal_index3}, where the analysis from Section \ref{sec_1dOptimalIndex} can be applied separately in the two parts. The total computational cost $W$ satisfies
\[
W= O\left(\varepsilon^{-2}\big|\log\varepsilon\big|^3\right).
\]

\begin{figure}[H]
\centering
\includegraphics[width=2.5in, height=2in]{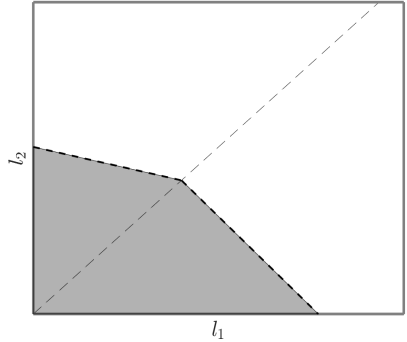}
\caption{Optimal index set, with schemes \eqref{eq_1dimplicitdiscrete_2} and \eqref{eq_1dcoarseloss}.}
\label{fig_1d_optimal_index3}
\end{figure}

\section{Numerical implementation}\label{sec_1dNumerical}
Here, we present numerical tests for the schemes above to illustrate the theoretical convergence results and compare the complexity among MLMC and two MIMC schemes. Moreover we solve the SPDE on an interval $(x_{\min},x_{\max})$ with zero boundary conditions. 
We choose the parameters $\mu=0.081,\ \rho=0.2$, $x_{\min} = -10,\ x_{\max} = 20$, and the initial density is $v(0,x)=\delta(x-x_0), \text{where } x_0=5$. 

Figure~\ref{fig_theta} shows the values of $\theta$ from \eqref{eq_beginningoftheta} corresponding to different correlation $\rho$, where $\theta$ is derived from the following equation
\begin{equation}\label{eq_1d_thetacalculation}
h\int_{\Omega_{\text{high}}} \bigg(\frac{1+\rho\gamma^2ku\cos^2\frac{h\gamma}{2}+\frac{\rho^2}{2}\gamma^4k^2u^2\cos^4\frac{h\gamma}{2}}{1+\gamma^2ku+\frac{1}{4}\gamma^4k^2u^2}\bigg)^N\,\mathrm{d}\gamma = \theta^N.
\end{equation}
The proof of \eqref{eq_1d_thetacalculation} is given in Appendix \ref{appendix_proof_prop_meansquareconvergence}.

\begin{figure}[H]
\centering
\includegraphics[width=2.5in, height=2in]{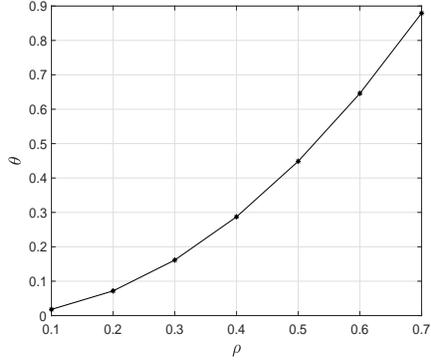}
\caption{$\theta$ vs. $\rho$}
\label{fig_theta}
\end{figure}

For $\rho=0.2$ and $T=5$, $\theta = 0.0678$. So if $h_0=1$, we can choose $k_0=1/2$ for $l_1^\ast \leq 12$, and $k_0=1/4$ for $l_1^\ast\leq 25$, as then $\theta^{N_0}\leq 2^{-C_0}\,h^{4}$, where $h=h_0\cdot2^{-l_1^\ast}$ and $N_0=T/k_0$. Therefore $k_0$ does not matter much in this problem. 

\subsection{Mean and variance of $\Delta L_{(l_1,\,l_2)}$}\label{sec_5.1}
In this section, we conduct numerical tests for $\mathbb{E}\Big[\Delta L_{(l_1,\,l_2)}\Big]$ and $\mathrm{Var}\Big[\Delta L_{(l_1,\,l_2)}\Big]$. The analysis shows that for \eqref{eq_1dimplicitdiscrete} and \eqref{eq_1dLossApprox},
\begin{align*}
\mathbb{E}\Big[\Delta L_{(l_1\,l_2)}\Big] = O(h^2k),\qquad
\mathrm{Var}\Big[\Delta L_{(l_1,\,l_2)}\Big] = O(h^4k^2),
\end{align*}
where
$$ h = h_0\cdot2^{-l_1},\quad k = k_0\cdot2^{-2l_2},$$
such that 
$$ O\big(h^2k\big)= 2^{-2(l_1+l_2)},\quad O\big(h^4k^2\big)= 2^{-4(l_1+l_2)}.$$

Table~\ref{table_1dmean} shows $\log_2\big|\mathbb{E}\Delta L_{(l_1,\,l_2)}\big|$ with different levels of timestep and mesh size, using $10^4$ Monte Carlo samples. 

The data in Table~\ref{table_1dmean} suggest that for $l_1^\ast = 25$ $k_0=1/4$, from given $l \coloneqq l_1+l_2$ to $l+1$ (i.e., from one diagonal to the next), the values decrease by around 2, in line with \eqref{eq_1dassumptions_1}. Figure~\ref{fig_1dmeanplot} is the contour plot of the values of $\log_2\big|\mathbb{E}\Delta L_{(l_1,\,l_2)}\big|$. 
\begin{table}[H]
\center
{\begin{tabular}{@{}c|ccccc}
\backslashbox{$l_1$ ($h$)}{$l_2$ ($k$)} & \multicolumn{1}{c}{1} & \multicolumn{1}{c}{2} & \multicolumn{1}{c}{3} & \multicolumn{1}{c}{4} & \multicolumn{1}{c}{5}  \\
\hline
$1$ & -14.3384 & -16.2991 & -18.2382 & -20.1987 & -22.2193 \\
$2$ & -16.2725 & -18.3163 & -20.2538 & -22.1928 & -24.2061 \\
$3$ & -18.2749 & -20.3508 & -22.2862 & -24.2709 & -26.2991\\
$4$ & -20.2770 & -22.3616 & -24.2059 & -26.2203 & -28.3113\\
$5$ & -22.2776 & -24.3644 & -26.2984 & -28.2543 & -30.2064\\
\end{tabular}}
\caption{$\log_2\big(\big|\mathbb{E}\Delta L_{(l_1,\,l_2)}\big|\big)$ with $h = h_0\cdot2^{-l_1}$ and $k = k_0\cdot2^{-2l_2}$, for approximation \eqref{eq_1dimplicitdiscrete} and \eqref{eq_1dLossApprox}.}
\label{table_1dmean}
\end{table}

Table~\ref{table_1dsquare} shows the corresponding $\log_2\big(\mathrm{Var}\big[\Delta L_{(l_1\,l_2)} \big]\big)$. We can see from the table that from given $l \coloneqq l_1+l_2$ to $l+1$, the values decrease by approximately 4, in line with \eqref{eq_1dassumptions_2}. Figure~\ref{fig_1dvarplot} is the contour plot of the values of $\log_2\mathrm{Var}[\Delta L_{(l_1\,l_2)}]$.
\begin{table}[H]
\center
{\begin{tabular}{@{}c|cccccc}
\backslashbox{$l_1$ ($h$)}{$l_2$ ($k$)} & \multicolumn{1}{c}{1} & \multicolumn{1}{c}{2} & \multicolumn{1}{c}{3} & \multicolumn{1}{c}{4} & \multicolumn{1}{c}{5}\\
\hline
$1$ & -26.2614 & -30.0205 & -34.0963 & -38.1821 & -42.0136 \\
$2$ & -29.3889 & -33.1355 & -37.2427 & -41.3587 & -45.2778 \\
$3$ & -33.1243 & -36.8837 & -41.0030 & -44.9869 & -48.9696 \\
$4$ & -37.0547 & -40.8191 & -44.9404 & -48.9230 & -52.9080 \\
$5$ & -41.0371 & -44.8028 & -48.8718 & -52.9715 & -56.9993 \\ 
\end{tabular}}
\caption{$\log_2\big(\mathrm{Var}[\Delta L_{(l_1\,l_2)}]\big)$ with $h = h_0\cdot2^{-l_1}$ and $k = k_0\cdot2^{-2l_2}$, for approximation \eqref{eq_1dimplicitdiscrete} and \eqref{eq_1dLossApprox}.}
\label{table_1dsquare}
\end{table}

\begin{figure}[H]
\centering
\subfigure[$\log_2\big|\mathbb{E}\Delta L_{(l_1\,l_2)} \big|$]{
\includegraphics[width=2.5in, height=2in]{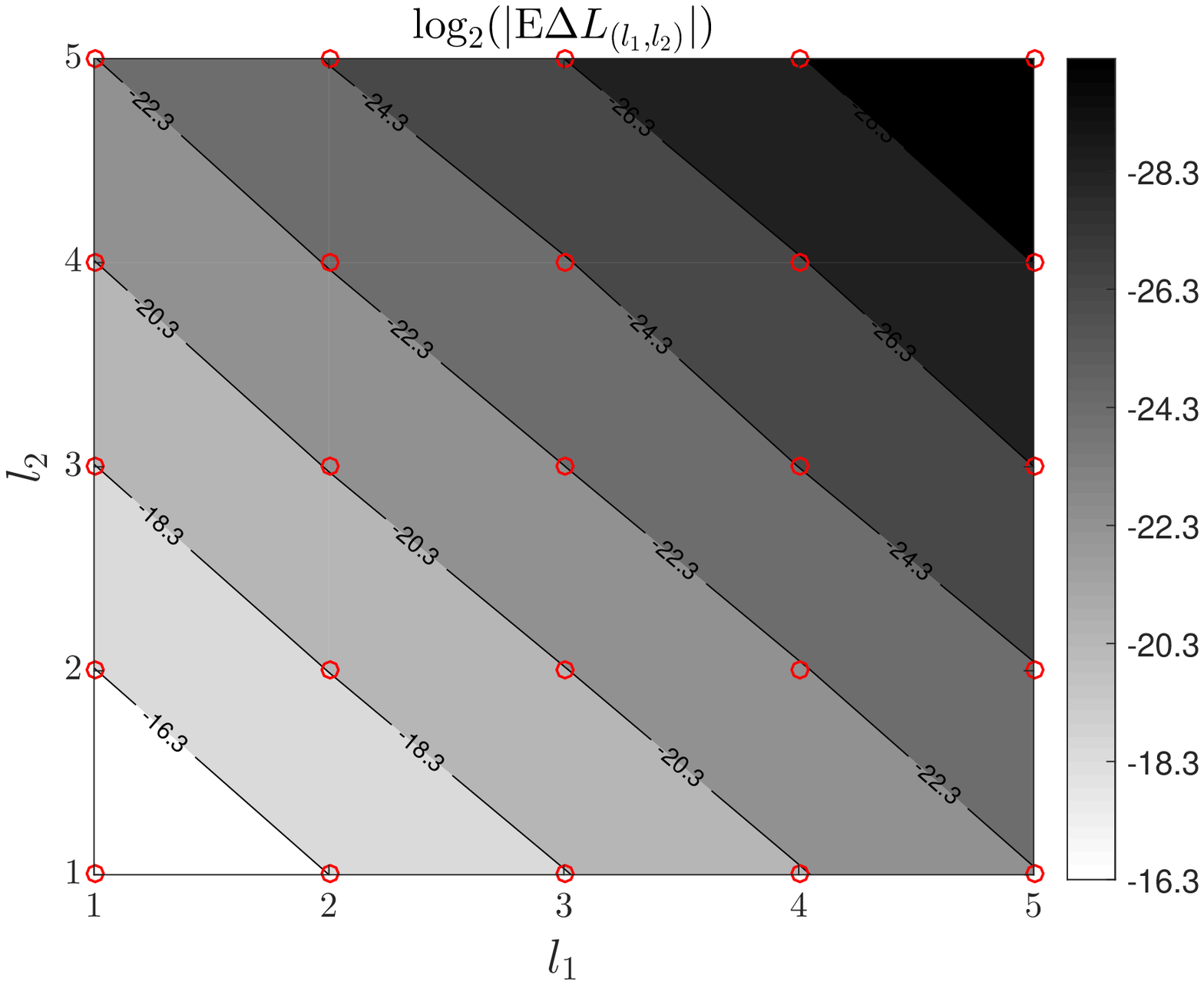}
\label{fig_1dmeanplot}
}
\subfigure[$\log_2\mathrm{Var}(\Delta L_{(l_1,\,l_2)})$]{
\includegraphics[width=2.5in, height=2in]{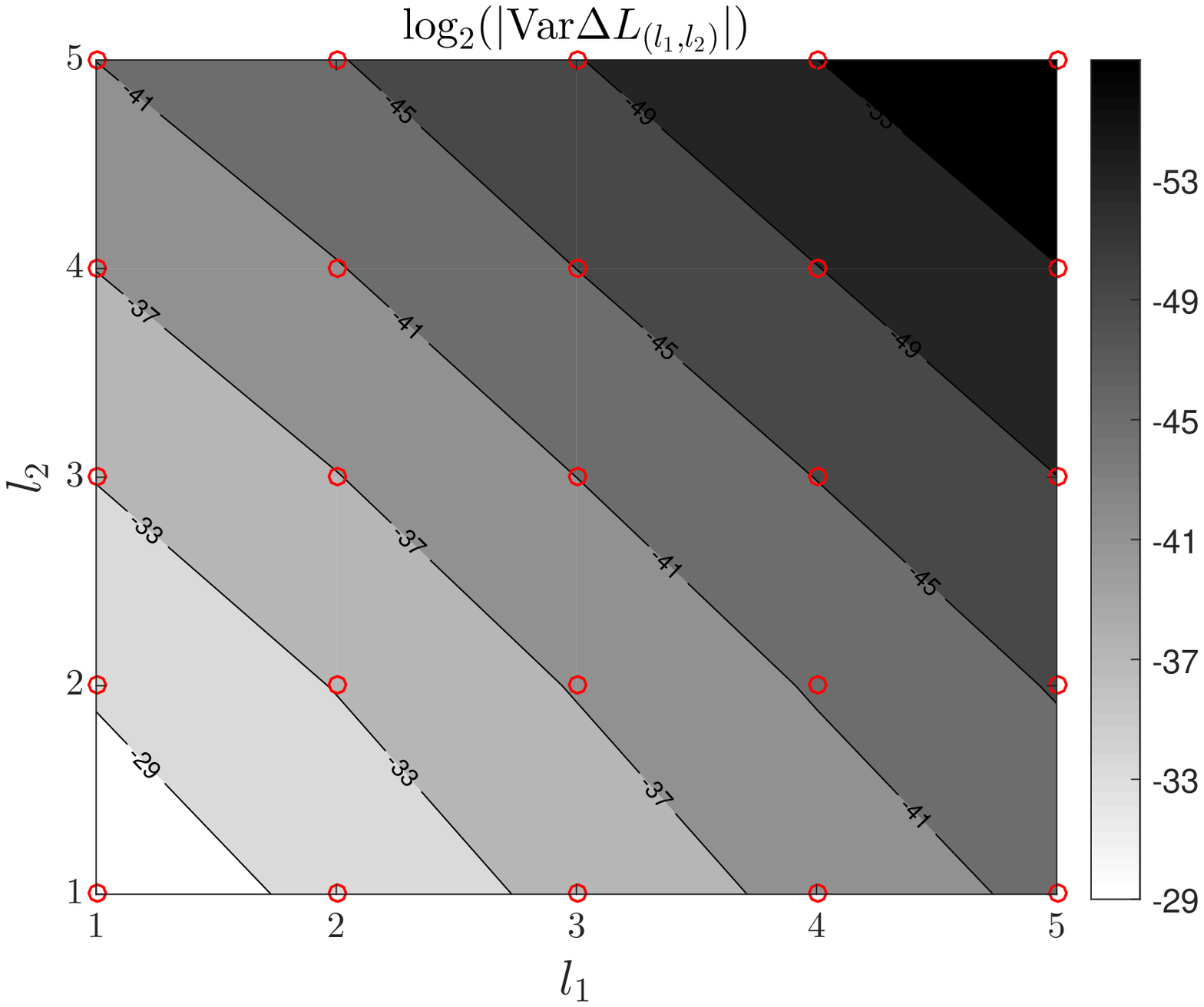}
\label{fig_1dvarplot}}
\caption{Rate verification: contour plots of log values of sample mean (left) and variance (right) of mixed differences used in MIMC.}
\end{figure}
Figure~\ref{fig_1dboundaryplot} compares MLMC and MIMC at boundary levels, i.e., $l_1=0$ or $l_2=0$. We can see from the plots that they have the same rate of convergence, which is consistent with Remark~\ref{rmk_1dfirstorderdifference}.
\begin{figure}[H]
\centering
\subfigure[$\log_2\big|\mathbb{E}\Delta L_{(l_1\,l_2)} \big|$]{
\includegraphics[width=2.5in, height=2in]{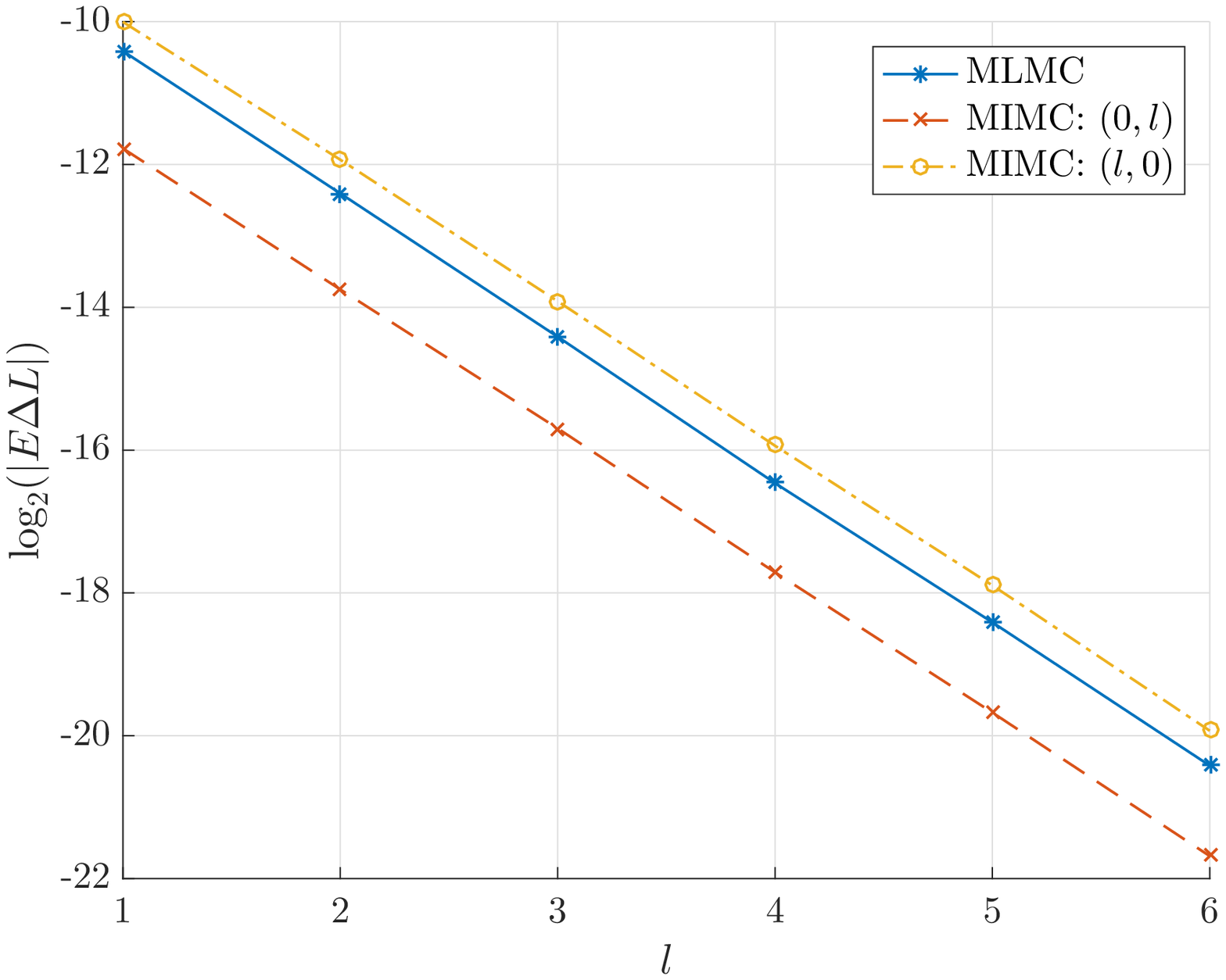}
\label{fig_1dboundarymeanplot}
}
\subfigure[$\log_2\mathrm{Var}(\Delta L_{(l_1,\,l_2)})$]{
\includegraphics[width=2.5in, height=2in]
{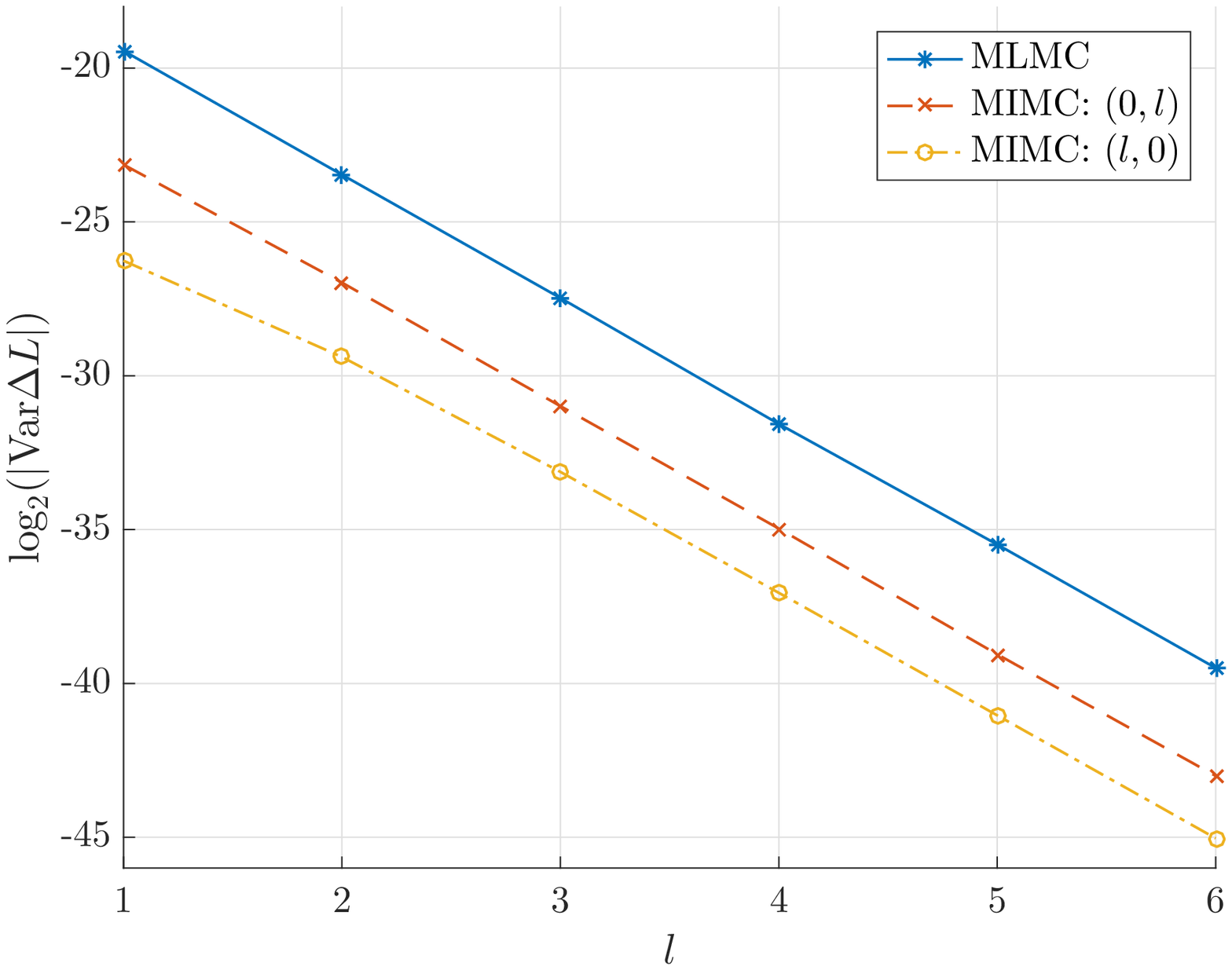}
\label{fig_1dboundaryvarplot}}
\caption{Log values of mean and variance for MLMC and MIMC with $l_1=0$ or $l_2=0$.}
\label{fig_1dboundaryplot}
\end{figure}

Now we compare the results with the alternative discretisation scheme \eqref{eq_1dimplicitdiscrete_2} but the same $L_{(l_1,\,l_2)}$ from \eqref{eq_1dLossApprox}. The analysis in Section \ref{sec_1dalternative} shows that under that scheme,
\begin{align*}
\mathbb{E}\Big[\Delta L_{(l_1\,l_2)}\Big] = O(h^2k),\qquad
\mathrm{Var}\Big[\Delta L_{(l_1,\,l_2)}\Big] = O(h^4k).
\end{align*}

Figure~\ref{fig_oldscheme} shows the contour plots of $\log_2\big|\mathbb{E}\Delta L_{(l_1,\,l_2)}\big|$ and $\log_2\big(\mathrm{Var}[\Delta L_{(l_1\,l_2)}]\big)$ using the alternative discretisation scheme \eqref{eq_1dimplicitdiscrete_2} and \eqref{eq_1dLossApprox}, with $h = h_0\cdot2^{-l_1}$ and $k = k_0\cdot2^{-2l_2}$.

\begin{figure}[H]
\centering
\subfigure[$\log_2\big|\mathbb{E}\Delta L_{(l_1\,l_2)} \big|$]{
\includegraphics[width=2.5in, height=2in]{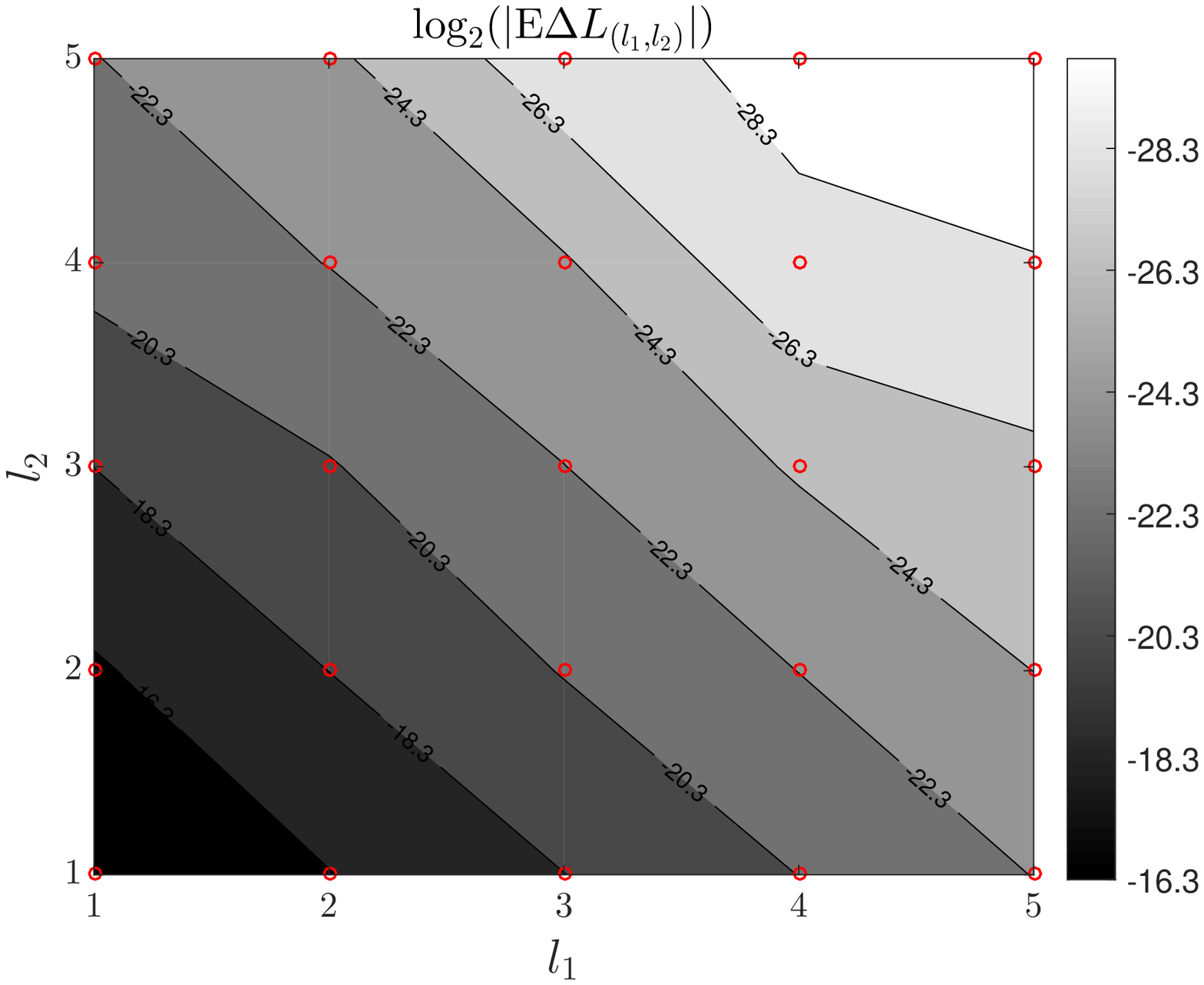}
\label{fig_1dmeanplot2}
}
\subfigure[$\log_2\mathrm{Var}\left(\Delta L_{(l_1,\,l_2)}\right)$]{
\includegraphics[width=2.5in, height=2in]{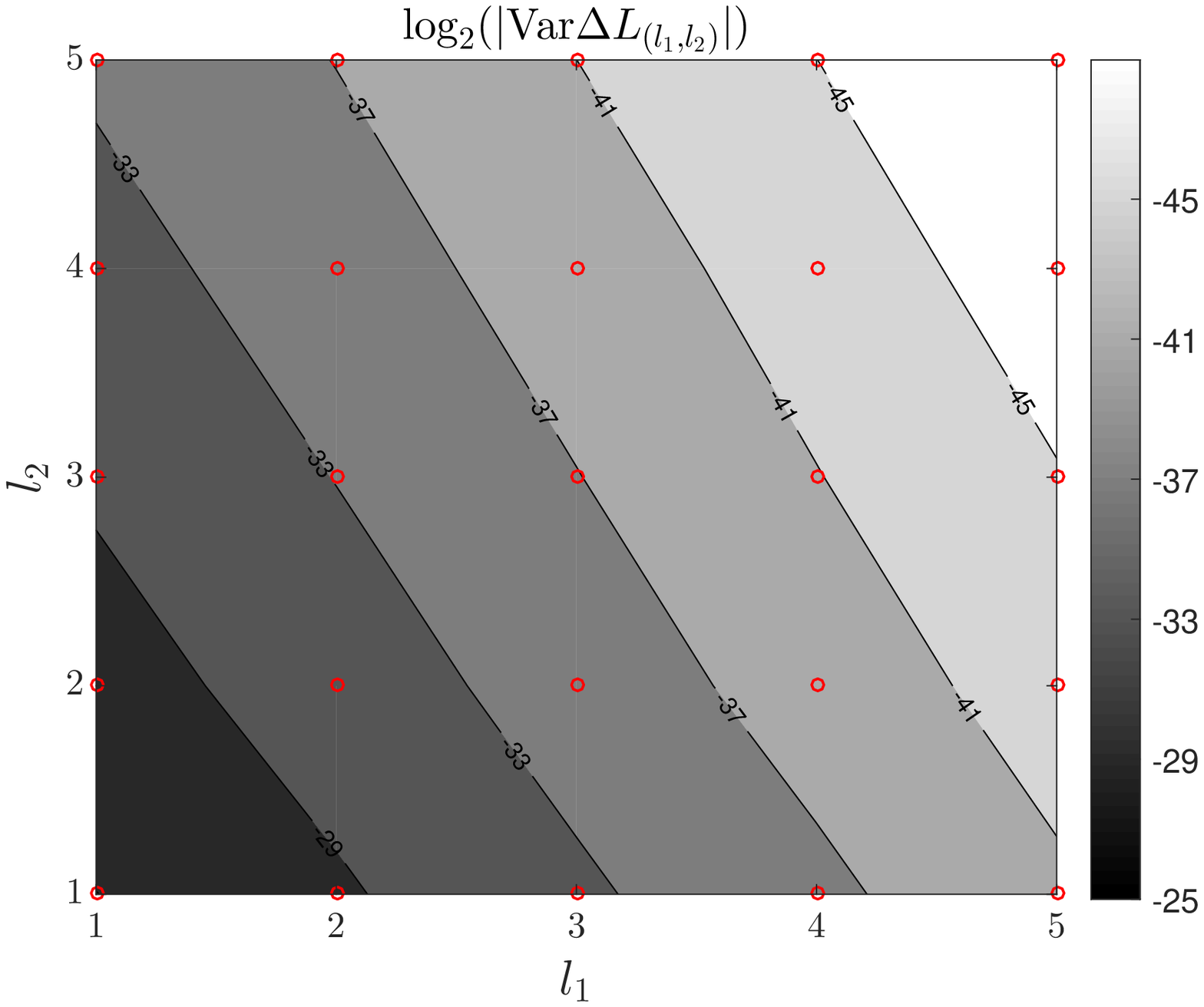}
\label{fig_1dvarplot2}}
\caption{Log values of mean and variance by the other discretisation scheme, given by \eqref{eq_1dimplicitdiscrete_2} and \eqref{eq_1dLossApprox}.}
\label{fig_oldscheme}
\end{figure}

We compute the optimal index set based on the profit \eqref{eq_profit}. Figure \ref{fig_1dprofit1} shows the profit using scheme \eqref{eq_1dimplicitdiscrete} and \eqref{eq_1dLossApprox}. This provides the basis for the optimal index set \eqref{eq_1dindexset_case1}. For scheme \eqref{eq_1dimplicitdiscrete_2} and \eqref{eq_1dcoarseloss}, there are two leading orders of variance, hence the index set is not triangular, as one can deduce from Figure \ref{fig_1dprofit3}.

\begin{figure}[H]
\centering
\subfigure[Profit of scheme \eqref{eq_1dimplicitdiscrete} and \eqref{eq_1dLossApprox}.]{
\includegraphics[width=2.5in, height=2in]{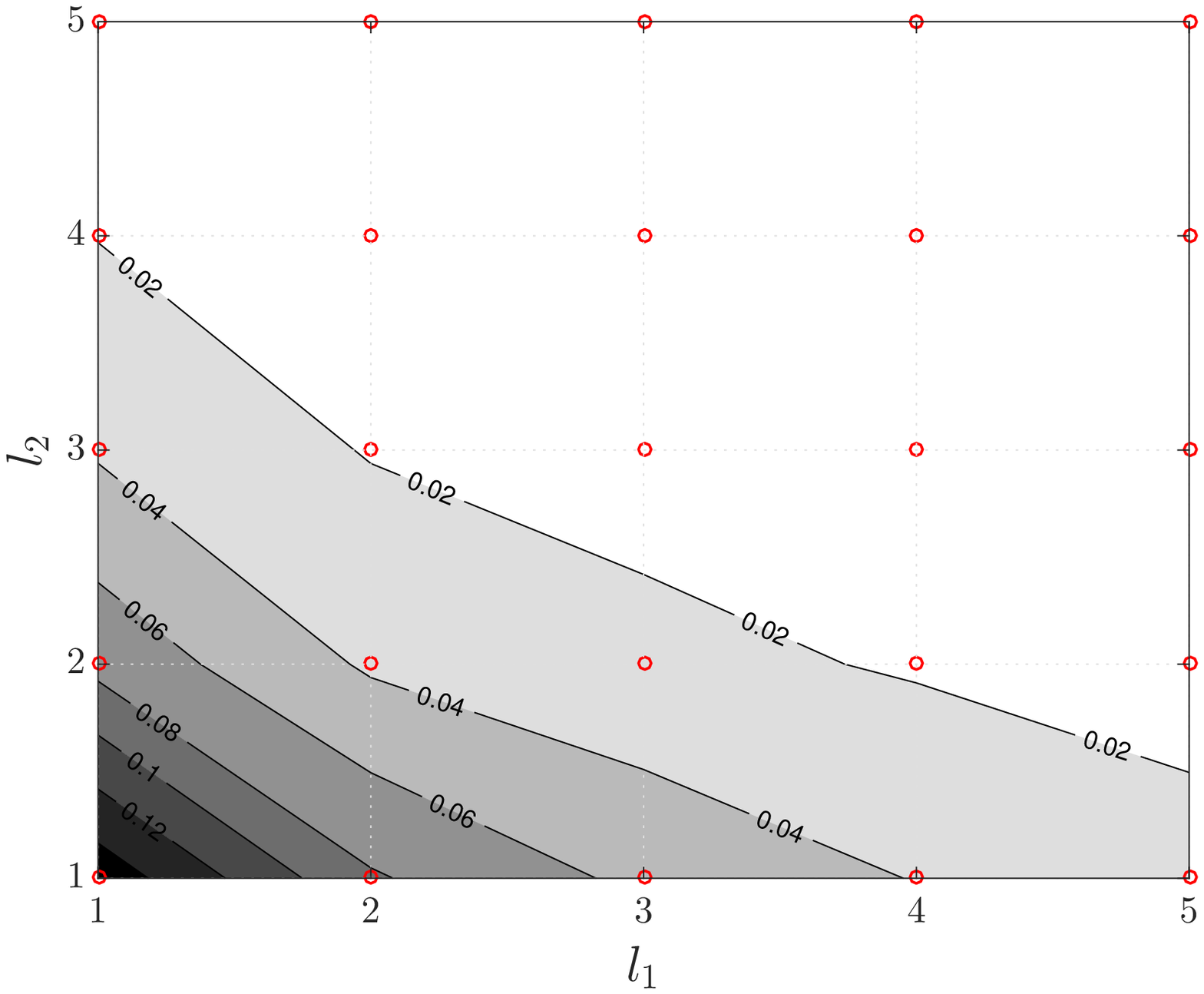}
\label{fig_1dprofit1}
}
\subfigure[Profit of scheme \eqref{eq_1dimplicitdiscrete_2} and \eqref{eq_1dcoarseloss}.]{
\includegraphics[width=2.5in, height=2in]{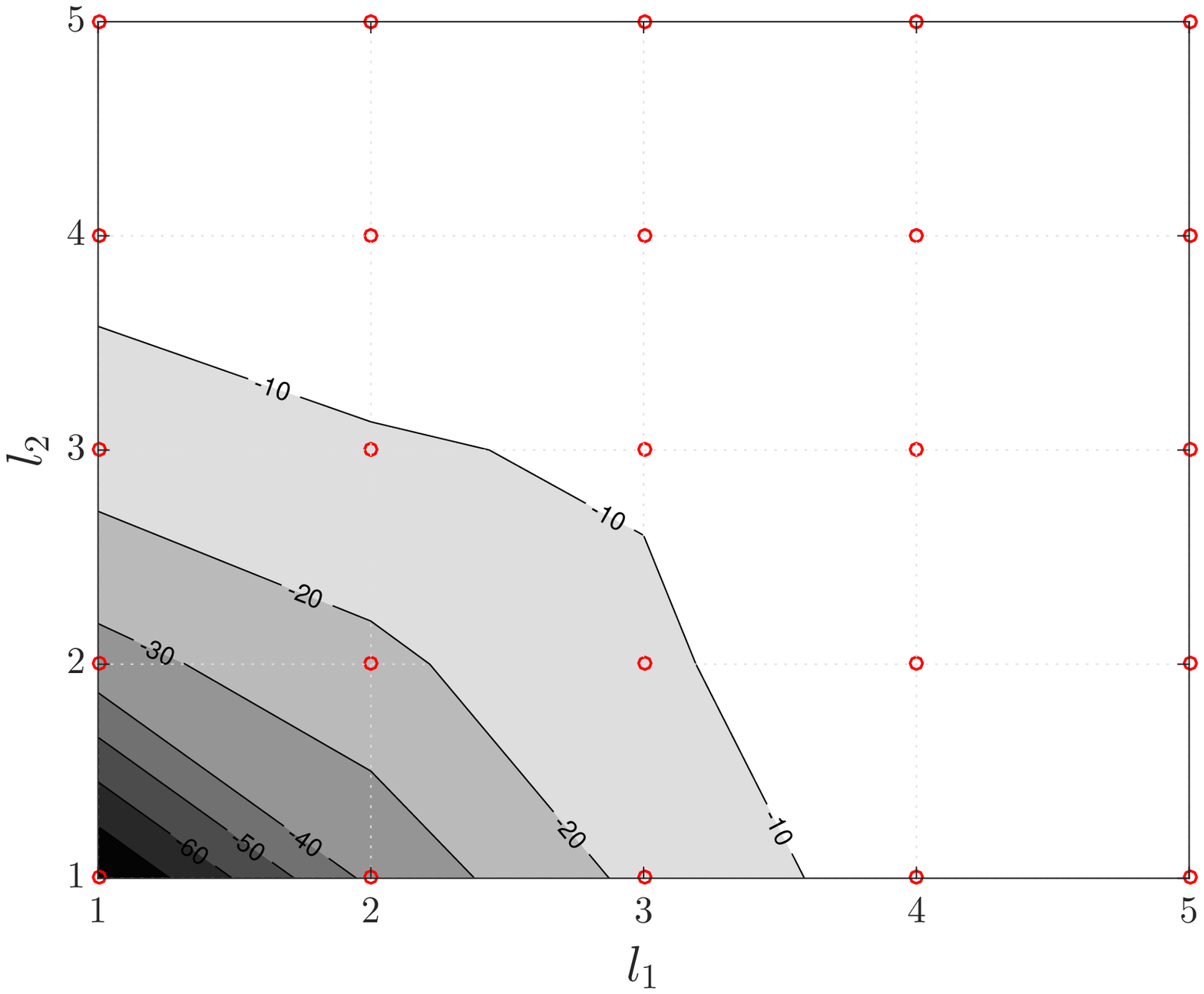}
\label{fig_1dprofit3}}
\caption{Contour plots of profit.}
\label{fig_1dprofit}
\end{figure}

\subsection{Complexity}
The analysis shows that the total computational cost of the scheme \eqref{eq_1dimplicitdiscrete} and \eqref{eq_1dLossApprox} for a RMSE $\varepsilon$ is $W = O\left(\varepsilon^{-2}\big|\log\varepsilon\big|\right)$. As $k_0$ does not matter much in this problem, we can choose $k_0=1/4$ fixed for $l_1^\ast\leq 25$, and the computational cost is expected to have the order $O\left(\varepsilon^{-2}\right)$. We test the total cost given different accuracy~$\varepsilon$, and compare the results with the multilevel algorithm. The mean square error for the estimator can be expressed as the sum of the variance of the estimator and the square of the weak error, as in \eqref{eq_1drms},
$$\big|\mathbb{E}\big[\widehat{L}_{\mathcal{I}}-L\big]\big| \leq \alpha\varepsilon, \qquad \mathrm{Var}\big[\widehat{L}_{\mathcal{I}}\big] \leq \big(1-\alpha^2\big)\varepsilon^2.$$
Since
$V_{(l_1,\,l_2)}W_{(l_1,\,l_2)}\sim 2^{-3l_1-2l_2}$,
the variance decays more rapidly with the level than the cost increases, so that the dominant cost is on level~$0$. For each fixed accuracy, we find the global minimum of the total cost with respect to $\alpha$ and choose that optimal~$\alpha^\ast$, thus reducing the total cost. Figure~\ref{fig_alpha} shows how the total cost varies with $\alpha$ when $\varepsilon=10^{-4}$. Figure~\ref{fig_1d_comparecost2} plots the CPU time of $\varepsilon^2W$ of two multi-index and the multilevel algorithms. As $k_0$ does not affect much in this problem, the total computational costs of the MLMC algorithm and MIMC with schemes \eqref{eq_1dimplicitdiscrete} and \eqref{eq_1dLossApprox} are approximately proportional to $\varepsilon^{-2}$, hence $\varepsilon^2W$ should not vary much with different accuracy $\varepsilon$. However, the cost of the alternative discretisation scheme \eqref{eq_1dimplicitdiscrete_2} with MIMC has the order $\varepsilon^{-2}(\log\varepsilon)^2$. We can see from the figure that $\varepsilon^2 W$ is no longer a constant but is slightly increasing as $\varepsilon\rightarrow 0$. 

\begin{figure}[H]
\centering
\subfigure[Cost versus $\alpha$ with scheme \eqref{eq_1dimplicitdiscrete} and \eqref{eq_1dLossApprox}.]{
\includegraphics[width=2.5in, height=2in]{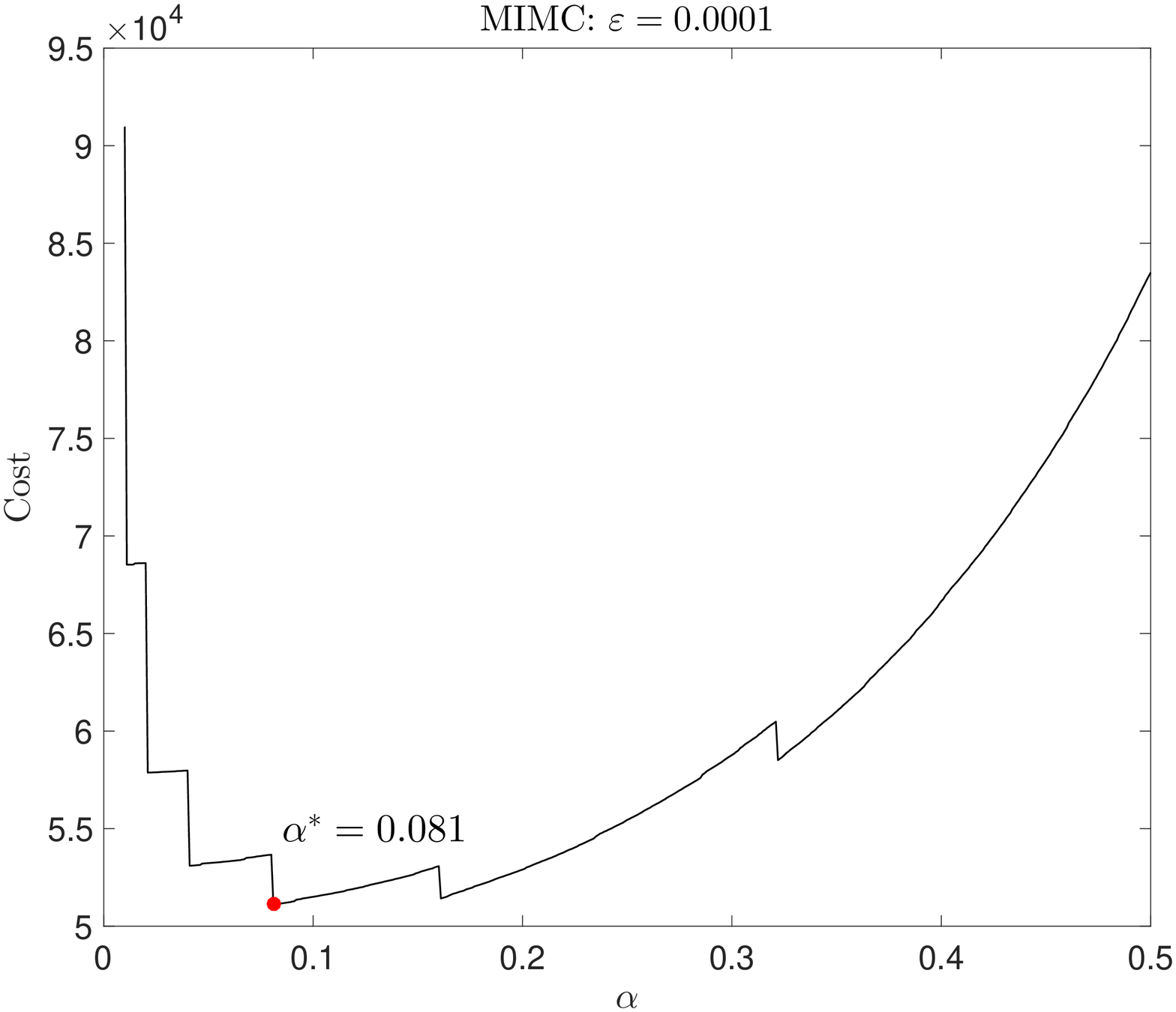}
\label{fig_alpha}
}
\subfigure[$\varepsilon^2$Cost of different schemes.]{
\includegraphics[width=2.5in, height=2in]{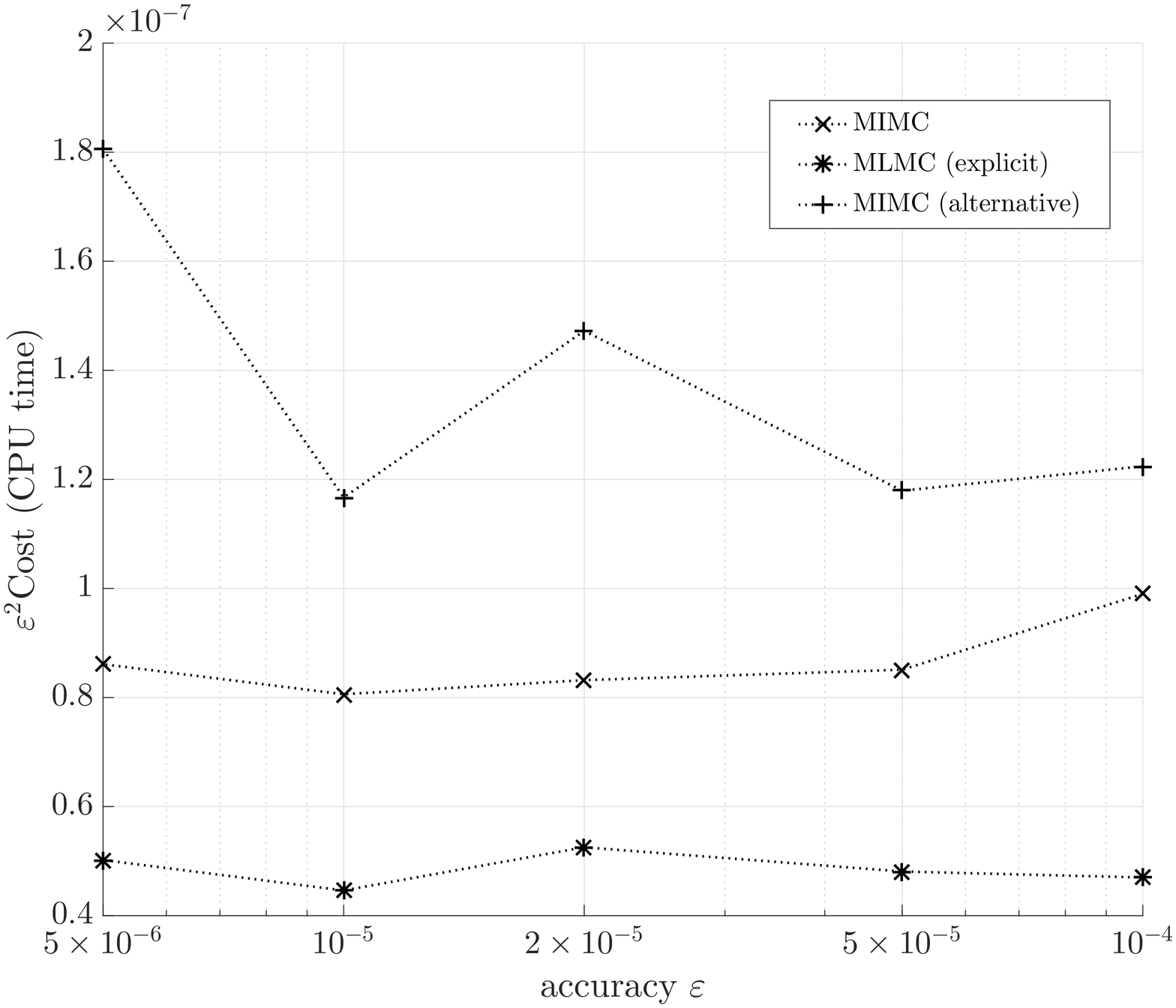}
\label{fig_1d_comparecost2}}
\caption{Total cost versus $\alpha$, and $\varepsilon^2$Cost with different schemes, where ``alternative'' denotes the scheme \eqref{eq_1dimplicitdiscrete_2} and \eqref{eq_1dLossApprox}.}
\end{figure}

The costs of the MLMC and MIMC schemes are very similar across a wide range of parameters (not reported here), which is to be expected as the dominant cost comes from the coarsest level. The MLMC scheme is slightly more efficient here, because it allows explicit time-stepping with a slightly lower computational cost. When an implicit scheme \eqref{eq_1dimplicitdiscrete} is used for MLMC as well, giving better stability properties, the MIMC scheme outperforms. This is useful for locally refined meshes.

\section{Conclusions and further work}\label{sec_Conclusion}

We have analysed the accuracy and complexity of a MIMC estimator for a one-dimensional parabolic SPDE using Fourier analysis. Specifically, we analysed a functional $L\left(v(\cdot,T)\right)$ of the solution. We showed that, by using the implicit Milstein finite difference discretisation~\eqref{eq_1dimplicitdiscrete}, the order of the first and second moments of $\Delta L_{(l_1,\,l_2)}$ are $O(h^2k)$ and $O(h^4k^2)$, with $h=h_0\cdot 2^{-l_1},\ k=k_0\cdot 2^{-2l_2}$. For a fixed RMSE $\varepsilon$, the theoretical complexity is $O(\varepsilon^{-2}|\log\varepsilon|)$. However, practically the order of complexity is still $O(\varepsilon^{-2})$. Moreover, under a different discretisation \eqref{eq_1dimplicitdiscrete_2}, the first and second moments of $\Delta L_{(l_1,\,l_2)}$ are $O(h^2k)$ and $O(h^4k)$, and practically the total complexity of this scheme is $O\big(\varepsilon^{-2}\left(\log\varepsilon\right)^2\big)$. We further used a simpler approximation of $L$ as \eqref{eq_1dcoarseloss} with discretisation \eqref{eq_1dimplicitdiscrete_2}. This gave two leading terms of different orders in the variance, which violated the simplest form of assumptions of the framework in \cite{ref5}. After some adaptation, theoretically we achieved a complexity of $O\big(\varepsilon^{-2}\left|\log\varepsilon\right|^3\big)$.

Although MLMC has already achieved optimal complexity in this case, one shortcoming of MLMC is that the efficiency still depends on the dimensionality. For high-dimensional problems, when the level increases, the decay of the variance will be slower than the increase of the cost. In that case, the total computational cost is no longer $O(\varepsilon^{-2})$. For example, consider a two-dimensional version of the SPDE,
\begin{align*}\label{eq_2dSPDE}
\mathrm{d}v &= -\mu_x\frac{\partial v}{\partial x}\,\mathrm{d}t -\mu_y\frac{\partial v}{\partial y}\,\mathrm{d}t + \frac{1}{2}\bigg(\frac{\partial^2 v}{\partial x^2} + 2\sqrt{\rho_x\rho_y}\rho_{xy}\frac{\partial^2v}{\partial x\partial y} + \frac{\partial^2 v}{\partial y^2} \bigg)\,\mathrm{d}t\\
&\qquad - \sqrt{\rho_x}\frac{\partial v}{\partial x}\,\mathrm{d}M_t^x - \sqrt{\rho_y}\frac{\partial v}{\partial y}\,\mathrm{d}M_t^y,
\end{align*}
with $\rho_x, \rho_y \in [0,1)$, $\rho_{xy} \in [-1,1]$, $\mu_x,\mu_y$ real parameters.
Assuming second order convergence in both spatial directions and first order in time, the total MLMC cost given a fixed accuracy $\varepsilon$ is expected to be $O\left(\varepsilon^{-2}(\log\varepsilon)^2\right)$. Although this is much lower than the $O(\varepsilon^{-4})$ expected for the standard Monte Carlo method, the order is not optimal as in the one-dimensional case. Further work will include analysing high-dimensional SPDEs using the MIMC method. Preliminary results suggest that MIMC achieves a better complexity than the multilevel method.

Another further research is to apply an absorbing boundary condition to this Zakai type SPDE \cite{bush2011stochastic}. Now the particle system involves a barrier such that once the particle touches the barrier, the value would be zero afterwards. Then the density process of the system satisfies a SPDE with a zero boundary condition. Specifically, for the one-dimensional SPDE \eqref{eq_1dSPDE}, it becomes
\begin{equation}\label{eq_1dSPDE_boundary}
\begin{aligned}
\mathrm{d}v &= -\mu\frac{\partial v}{\partial x}\,\mathrm{d}t + \frac{1}{2}\frac{\partial^2 v}{\partial x^2}\,\mathrm{d}t - \sqrt{\rho}\frac{\partial v}{\partial x}\,\mathrm{d}M_t,\\
v(0,x) &= \delta(x_0),\\
v(t,0) &= 0.
\end{aligned}
\end{equation}
The solution has a smooth density process and shows degeneracy near the absorbing boundary \cite{seanboundarySPDEs}.
In this case, the convergence order of the finite difference scheme does not follow from the Fourier analysis of the unbounded case. The analysis of this type of SPDE with boundary conditions is still an open area for research, especially for higher dimensions. One possible way is to use combination technique which combines the Galerkin solutions on certain full tensor product spaces \cite{griebel2014convergence}.

\section*{Acknowledgements}
The authors would like to thank Raul Tempone, Abdul-Lateef Haji-Ali, and Mike Giles for helpful discussions on Multi-index Monte Carlo methods in this context.

\begin{appendix}

\section{Proof of Proposition~\ref{prop_1dImplicitConvergence}}\label{appendix_proof_prop_meansquareconvergence}
As for the mean-square convergence, Proposition \ref{prop_1dImplicitConvergence}, we compare the numerical solution to the analytical solution by splitting the domain into two wave number regions. Assume $p$ is a constant satisfying $0<p<\frac{1}{4}$. Then, we define the low wave number region as in \eqref{eq_1d_Omega_low},
\[
\Omega_{\text{low}} = \big\{\gamma: \vert\gamma\vert\leq\min\{h^{-2p},\,k^{-p}\}\big\},
\]
and the high wave number region as in \eqref{eq_1d_Omega_high},
\[
\Omega_{\text{high}} = \big\{\gamma: \vert\gamma\vert>\min\{h^{-2p},\,k^{-p}\}\big\}\cap \left[-\frac{\pi}{h},\frac{\pi}{h}\right].
\]

For $\gamma$ in the low wave region, the analysis is the same as \cite{ref1}, Theorem 2.1, that the Fourier transform of the numerical solution is very close to the theoretical solution, with error $O(h^2) + O(k)$. However, when $h^2<k$, there is an extra term $O(\theta^Nh^{-1})$ in the high wave region. So here we only analyse the contribution from the high wave region.

\begin{proof}[Proposition 2.1]\label{proof_prop_1dmeansquarehighwave}
For the case $h^2<k$, we divide the high wave region into two parts, $\{k^{-p}<|\gamma|<k^{-1/2}\}$ and $\{k^{-1/2}<|\gamma|<\pi/h\}$, and we integrate separately. Note that $X(T,\gamma) = \exp(-\frac{1}{2}(1-\rho)\gamma^2T-\mathrm{i}\gamma\sqrt{\rho}\,W_T)$.
\begin{align*}
&\quad \mathbb{E}\Bigg[\,\bigg|\int_{\Omega_{\text{high}}} X(T,\gamma)-X_N(\gamma)\,\mathrm{d}\gamma\bigg|^2\,\Bigg]\\
&\leq 4 k^{-1/2}\mathbb{E}\Bigg[\int_{k^{-p}}^{k^{-1/2}} \Big|X(T,\gamma)-X_N(\gamma)\Big|^2\,\mathrm{d}\gamma \Bigg] + 4\pi h^{-1}\mathbb{E}\Bigg[\int_{k^{-1/2}}^{\pi/h} \Big|X(T,\gamma)-X_N(\gamma)\Big|^2\,\mathrm{d}\gamma \Bigg]\\ 
&\leq 8k^{-1/2}\int_{k^{-p}}^{k^{-1/2}}  \mathbb{E}\left[\left|X_N(\gamma)\right|^2 + \left|X(T,\gamma)\right|^2\right]\,\mathrm{d}\gamma + 8\pi h^{-1}\int_{k^{-1/2}}^{\pi/h} \mathbb{E}\left[\left|X_N(\gamma)\right|^2 + \left|X(T,\gamma)\right|^2\right]\,\mathrm{d}\gamma \\
&= 8k^{-1/2}\int_{k^{-p}}^{k^{-1/2}} \bigg(\frac{1+\rho\gamma^2ku\cos^2\frac{h\gamma}{2}+\frac{\rho^2}{2}\gamma^4k^2u^2\cos^4\frac{h\gamma}{2}}{1+\gamma^2ku+\frac{1}{4}\gamma^4k^2u^2}\bigg)^N \,\mathrm{d}\gamma \\
&\quad + 8\pi h^{-1}\int_{k^{-1/2}}^{\pi/h} \bigg(\frac{1+\rho\gamma^2ku\cos^2\frac{h\gamma}{2}+\frac{\rho^2}{2}\gamma^4k^2u^2\cos^4\frac{h\gamma}{2}}{1+\gamma^2ku+\frac{1}{4}\gamma^4k^2u^2}\bigg)^N\,\mathrm{d}\gamma + f_0(k),
\end{align*}
where $u = \frac{\sin^2\frac{h\gamma}{2}}{(\frac{h\gamma}{2})^2}$, and
\begin{equation}\label{eq_1dappendix_f0k}
\begin{aligned}
f_0(k)&\coloneqq 8k^{-1/2}\int_{k^{-p}}^{k^{-1/2}}\exp\big(-(1-\rho)\gamma^2T\big)\,\mathrm{d}\gamma + 8\pi h^{-1}\int_{k^{-1/2}}^{\pi/h}\exp\big(-(1-\rho)\gamma^2T\big)\,\mathrm{d}\gamma\\
&\leq \frac{4k^{p-1/2}}{(1-\rho)T}\exp\big(-(1-\rho)Tk^{-2p}\big) + \frac{4\pi k^{1/2}}{(1-\rho)T\,h}\exp\big(-(1-\rho)Tk^{-1}\big).
\end{aligned}
\end{equation}
Denote $\lambda \coloneqq k/h^2\in[1,\infty)$, and 
\[
f(\gamma) \coloneqq \frac{1+\rho\gamma^2ku\cos^2\frac{h\gamma}{2}+\frac{\rho^2}{2}\gamma^4k^2u^2\cos^4\frac{h\gamma}{2}}{1+\gamma^2ku+\frac{1}{4}\gamma^4k^2u^2} = \frac{1+\rho\cdot \lambda \sin^2(h\gamma) +\frac{1}{2}\rho^2\lambda^2\sin^4(h\gamma)}{1+4\lambda\sin^2\frac{h\gamma}{2} + 4\lambda^2\sin^4\frac{h\gamma}{2}}.
\]
we have
\begin{align*}
f(\gamma)&\leq \frac{1+\rho\gamma^2ku +\frac{\rho^2}{2}\gamma^4k^2u^2}{1+\gamma^2ku+\frac{1}{4}\gamma^4k^2u^2}= \frac{1+\rho\cdot 4\lambda \sin^2\frac{h\gamma}{2} +8\rho^2\lambda^2\sin^4\frac{h\gamma}{2}}{1+4\lambda\sin^2\frac{h\gamma}{2} + 4\lambda^2\sin^4\frac{h\gamma}{2}}\coloneqq g(\gamma).
\end{align*}

So
\begin{align*}
g(\gamma) &= \frac{1+\rho\cdot 4\lambda \sin^2\frac{h\gamma}{2} +8\rho^2\lambda^2\sin^4\frac{h\gamma}{2}}{1+4\lambda\sin^2\frac{h\gamma}{2} + 4\lambda^2\sin^4\frac{h\gamma}{2}} = \frac{1+\rho\cdot 4\lambda d + 8\rho^2\lambda^2 d^2}{1+4\lambda d + 4\lambda^2 d^2} \\[5pt]
g'(\gamma) &= -\frac{4\rho(1-2\rho)\lambda^2 d^2 + 2(1-2\rho^2)\lambda d + (1-\rho)}{(1+4\lambda d + 4\lambda^2 d^2)^2}\cdot 2\lambda h\sin(\gamma h),
\end{align*}
where $d = \sin^2\frac{h\gamma}{2}$. Therefore we have
\[
g'(\gamma^\ast) = 0 \quad \Longleftrightarrow\quad 4\rho(1-2\rho)\lambda^2 d^2 + 2(1-2\rho^2)\lambda d + (1-\rho)=0.
\]
The equation has solutions $\gamma_1^\ast,\,\gamma_2^\ast$ as follows:
\begin{align}
  \sin^2\frac{h\gamma_1^\ast}{2} &= -\frac{1}{2\lambda} <0 \quad \mbox{(not possible)}, \label{eq_gamma_ast_1}\\
  \sin^2\frac{h\gamma_2^\ast}{2} &= \frac{1-\rho}{2\rho(2\rho-1)\lambda}\quad \mbox{(need $\rho>1/2$ and $\gamma^\ast<\frac{\pi}{h}$)}, \label{eq_gamma_ast_2}\\
  \sin^2\frac{h\gamma^\ast}{2} &= -\frac{1-\rho}{2(1-2\rho^2)\lambda}<0\quad \mbox{(for $\rho=1/2$ and this is not possible.)} \label{eq_gamma_ast_3}
\end{align}


So for $\gamma\in\Omega_{\mathrm{high}}$ and $0<\rho \leq1/2$, $g(\gamma)$ decreases over the high wave region:
\[
g'(\gamma)<0.
\]
\begin{enumerate}
\item
For $k^{-p}<\gamma<k^{-1/2}$,
\begin{align*}
g(\gamma) < g(k^{-p})  &= \frac{1+\rho\cdot 4\lambda \sin^2\frac{h}{2k^{p}} +8\rho^2\lambda^2\sin^4\frac{h}{2k^{p}}}{1+4\lambda\sin^2\frac{h}{2k^{p}} + 4\lambda^2\sin^4\frac{h}{2k^{p}}}\\
&=\frac{1+\rho\cdot 4\lambda \frac{h^2}{4k^{2p}} + (8\rho^2\lambda^2 - \frac{4}{3}\rho\lambda)\frac{h^4}{16k^{4p}} + O(\frac{h}{k^{2p}})^6}{1+4\lambda \frac{h^2}{4k^{2p}} + (4\lambda^2 - \frac{4}{3}\lambda)\frac{h^4}{16k^{4p}}+O(\frac{h}{k^{2p}})^6}\\
&= 1-\frac{(1-\rho)k^{1-2p} + \frac{1}{4}(1-2\rho^2)k^{2-4p}+\frac{1}{12}(1-\rho)hk^{1-4p}}{1+k^{1-2p} + \frac{1}{4}k^{2-4p}-\frac{1}{12}\rho hk^{1-4p}} + o(k^{3(1-4p)}),\\[5pt]
g^N(\gamma) < g^N(k^{-p}) &= \left(1-\frac{(1-\rho)k^{1-2p} + \frac{1}{4}(1-2\rho^2)k^{2-4p}+\frac{1}{12}(1-\rho)hk^{1-4p}}{1+k^{1-2p} + \frac{1}{4}k^{2-4p}-\frac{1}{12}\rho hk^{1-4p}} + o(k^{3(1-4p)})\right)^N\\
&< \exp\left( -(1-\rho)T \Big(k^{-2p}+\frac{1}{4}(1-2\rho^2)k^{-4p} + \frac{1}{12}(1-\rho)hk^{-4p} \Big) + o(k^{-4p})\right).
\end{align*}
Therefore,
\[
8k^{-1/2}\int_{k^{-p}}^{k^{-1/2}} g^N(\gamma)\,\mathrm{d}\gamma < 8k^{-1} \exp\left( -(1-\rho)T \Big(k^{-2p}+\frac{1-2\rho^2}{4}k^{-4p} + \frac{1-\rho}{12}hk^{-4p} \Big) + o(k^{-4p}) \right).
\]
This is $o(k^r)$ for any $r>0$ as $k\rightarrow 0$, i.e.,
\[
\lim_{k\rightarrow 0}\frac{1}{k^r}\cdot 8k^{-1/2}\int_{k^{-p}}^{k^{-1/2}} g^N(\gamma)\,\mathrm{d}\gamma =0,\qquad \mbox{for all $r>0$.}
\]
\item
For $k^{-1/2}<\gamma<\pi/h$,
\begin{align*}
g(\gamma) < g(k^{-1/2}) &= \frac{1+\rho \tilde{u} + \frac{1}{2}\rho^2 \tilde{u}^2}{1+ \tilde{u} + \frac{1}{4} \tilde{u}^2}\coloneqq q(\rho,\tilde{u}),\qquad\mbox{where $\tilde{u} = \bigg(\sin\big(\frac{h}{2\sqrt{k}}\big)\Big/\frac{h}{2\sqrt{k}}\bigg)^2 \in [4/\pi^2,1]$. }
\end{align*}
The function $q(\rho,\tilde{u})$
is strictly decreasing in $\tilde{u}$ and increasing in $\rho$ for $4/\pi^2\leq \tilde{u}\leq 1,\ 0<\rho\leq 1/2$.
So we have,
\begin{align*}
g(k^{-1/2}) < q\Big(\frac{1}{2},\frac{4}{\pi^2}\Big) = \frac{1+\frac{2}{\pi^2}+\frac{2}{\pi^4}}{1+\frac{4}{\pi^2}+\frac{4}{\pi^4}}\coloneqq \theta_0 < 0.8457.
\end{align*}
Therefore,
\[
8\pi h^{-1}\int_{k^{-1/2}}^{\pi/h} g^N(\gamma)\,\mathrm{d}\gamma < 8\pi^2 h^{-2}\cdot \theta_0^N.
\]
\end{enumerate}
Consequently, for $0<\rho\leq 1/2$,
\begin{align*}
&\quad\mathbb{E}\Bigg[\,\bigg|\int_{\Omega_{\text{high}}} X(T,\gamma)-X_N(\gamma)\,\mathrm{d}\gamma\bigg|^2\,\Bigg]\\
&\leq 8k^{-1/2}\int_{k^{-p}}^{k^{-1/2}} f^N(\gamma) \,\mathrm{d}\gamma + 8\pi h^{-1}\int_{k^{-1/2}}^{\pi/h} f^N(\gamma)\,\mathrm{d}\gamma + f_0(k)\\
&\leq 8k^{-1/2}\int_{k^{-p}}^{k^{-1/2}} g^N(\gamma) \,\mathrm{d}\gamma + 8\pi h^{-1}\int_{k^{-1/2}}^{\pi/h} g^N(\gamma)\,\mathrm{d}\gamma +f_0(k)\\
&\leq C\cdot \theta_0^{N}\cdot h^{-2},
\end{align*}
where $C$ and $\theta_0$ are constants independent of $h$ and $k$.

For $\gamma\in\Omega_{\mathrm{high}}$ and $1/2<\rho <1/\sqrt{2}$, $g(\gamma)$ is decreasing before $\gamma^\ast$, and increasing afterwards,
where $\gamma^\ast$ solves the equation \eqref{eq_gamma_ast_2} such that
\[
{\gamma^\ast}^2k = \frac{2(1-\rho)}{\rho(2\rho-1)}\cdot\left(\frac{\sin \frac{h\gamma^\ast}{2}}{\frac{h\gamma^\ast}{2}}\right)^{-2}\geq \frac{2(1-\rho)}{\rho(2\rho-1)}.
\]
So 
\[
\gamma^\ast \geq C_0\cdot k^{-1/2},\qquad\mbox{where $C_0 = \frac{2(1-\rho)}{\rho(2\rho-1)}\geq 2$.}
\]
We analyse separately the two parts, $\{k^{-p}<|\gamma|<k^{-1/2}\}$ and $\{k^{-1/2}<|\gamma|<\pi/h\}$. 
\begin{enumerate}
\item
When $k^{-p}<|\gamma|<k^{-1/2}$, $g(\gamma)$ is strictly decreasing. Similar to the previous analysis, we have 
\[
8k^{-1/2}\int_{k^{-p}}^{k^{-1/2}} g^N(\gamma)\,\mathrm{d}\gamma < 8k^{-1}\exp\left( -(1-\rho)T \Big(k^{-2p}+\frac{1}{4}(1-2\rho^2)k^{-4p} + \frac{1}{12}(1-\rho)hk^{-4p} \Big)\right).
\]
This is $o(k^r)$ for any $r>0$ as $k\rightarrow 0$.
\item
For $k^{-1/2}<|\gamma|<\pi/h$, $g(\gamma)$ decreases first, then increases if $\gamma^\ast<\pi/h$. Therefore,
\[
g(\gamma)<\max\{\,g(k^{-1/2}),\,g(\pi/h)\}.
\]
We will show both terms are strictly less than 1.

Similar to the previous analysis,
\begin{align*}
g(k^{-1/2}) = \frac{1+\rho\tilde{u} + \frac{1}{2}\rho^2\tilde{u}^2}{1+ \tilde{u} + \frac{1}{4} \tilde{u}^2} = q(\rho,\tilde{u}) <q\bigg(\frac{1}{\sqrt{2}},\frac{4}{\pi^2}\bigg)\coloneqq \theta_1 < 0.918.
\end{align*}
On the other hand,
\[
g(\pi/h) = \frac{1+4\rho\lambda+8\rho^2\lambda^2}{1+4\lambda+4\lambda^2},
\]
and this is strictly decreasing in $\lambda$ and increasing in $\rho$, so 
\[
g(\pi/h) \leq \frac{1+4\rho\lambda+8\rho^2\lambda^2}{1+4\lambda+4\lambda^2}\bigg|_{\lambda=1,\,\rho=1/\sqrt{2}} < 0.8699 \coloneqq \theta_2.
\]
Therefore, if we let $\theta_0 = \max\{\theta_1,\,\theta_2\} = \theta_1$, we have
\[
8\pi h^{-1}\int_{k^{-1/2}}^{\pi/h}  g^N(\gamma)\,\mathrm{d}\gamma < 8\pi^2 h^{-2}\cdot \theta_0^N.
\]
\end{enumerate}
Consequently, for $0<\rho\leq 1/2$,
\begin{align*}
&\quad\mathbb{E}\Bigg[\,\bigg|\int_{\Omega_{\text{high}}} X(T,\gamma)-X_N(\gamma)\,\mathrm{d}\gamma\bigg|^2\,\Bigg]\\
&\leq 8k^{-1/2}\int_{k^{-p}}^{k^{-1/2}} f^N(\gamma) \,\mathrm{d}\gamma + 8\pi h^{-1}\int_{k^{-1/2}}^{\pi/h} f^N(\gamma)\,\mathrm{d}\gamma + f_0(k)\\
&\leq 8k^{-1/2}\int_{k^{-p}}^{k^{-1/2}} g^N(\gamma) \,\mathrm{d}\gamma + 8\pi h^{-1}\int_{k^{-1/2}}^{\pi/h} g^N(\gamma)\,\mathrm{d}\gamma + f_0(k)\\
&\leq C\cdot \theta_0^{N}\cdot h^{-2},
\end{align*}
where $C$ and $\theta_0$ are constants independent of $h$ and $k$.


Numerical results confirm this as well. Figure~\ref{fig_fix_k} shows the integral over the high wave region with $T=1$, $k=1/4$ fixed, and $h=2^{-l}$, $l=7,8,\ldots,20$. 
\begin{figure}[H]
\centering
\includegraphics[width=2.5in, height=2in]{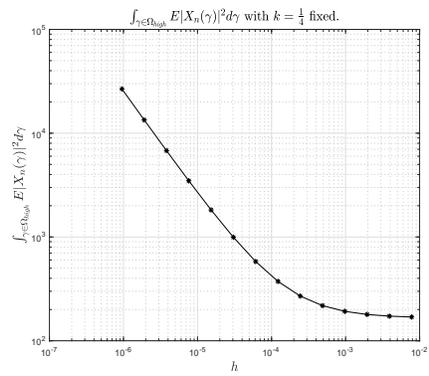}
\caption{Integral over the high wave region with $k$ fixed.}
\label{fig_fix_k}
\end{figure}

\end{proof}

\end{appendix}

\end{document}